\documentclass[12pt,x11names]{article}
\usepackage{geometry,amsmath,amssymb,amsthm,graphicx,epsfig}
\usepackage{subfigure,enumerate,dcolumn,latexsym,epstopdf,color}
\usepackage[graph]{xy}
\usepackage[colorlinks=true,linkcolor=blue,citecolor=blue]{hyperref}
\usepackage[usenames,dvipsnames,svgnames,table]{xcolor}

\usepackage{booktabs}
\numberwithin{equation}{section}
\newtheorem{theorem}{Theorem}[section]

\newtheorem{corollary}{Corollary}[section]
\newtheorem{proposition}{Proposition}[section]
\newtheorem{lemma}{Lemma}[section]
\theoremstyle{definition}
\newtheorem{remark}{Remark}[section]

\newtheorem{assumption}{Assumption}

\newcommand{\ie}{{\em i.e.}, }
\newcommand{\s}{\sigma}
\newcommand{\eg}{{\em e.g.}, }
\newcommand{\cf}{\emph{cf.\ }}

\newcommand{\prox}{\mathrm{prox}} 
\newcommand{\nn}{\mathbb{N}} 
\newcommand{\rr}{\mathbb{R}} 
\newcommand{\norm}[1]{\left\Vert {#1} \right\Vert} 
\newcommand{\dom}[1]{\mathrm{dom}\,{#1}} 
\newcommand{\gr}[1]{\mathrm{graph}\,{#1}} 
\newcommand{\crit}[1]{\mathrm{crit}\,{#1}}

\newcommand{\dist}{\mathrm{dist}} 
\newcommand{\act}[1]{\left\langle {#1} \right\rangle} 
\newcommand{\seq}[2]{\left\{{#1}_{{#2}}\right\}_{{#2} \in \mathbb{N}}}
\newcommand{\Seq}[2]{\left\{{#1}^{{#2}}\right\}_{{#2} \in \mathbb{N}}}
\newcommand{\argmin}{\operatorname{argmin}}
\newcommand{\limitinf}[2]{\liminf_{{#1} \rightarrow {#2}}}
\newcommand{\bx}{{\bf x}}

\newcommand{\bz}{{\bf z}}
\newcommand{\bu}{{\bf u}}

\newcommand{\bw}{{\bf w}}
\newcommand{\bo}{{\bf 0}}

\newcolumntype{C}[1]{>{\centering\let\newline\\\arraybackslash\hspace{0pt}}m{#1}}

\title{Inertial Proximal Alternating Linearized Minimization (iPALM)
  for Nonconvex and Nonsmooth Problems} \date{} \author{Thomas
  Pock\footnote{Institute of Computer Graphics and Vision, Graz
    University of Technology, 8010 Graz, Austria. Digital Safety \&
    Security Department, AIT Austrian Institute of Technology GmbH,
    1220 Vienna, Austria. E-mail: \texttt{pock@icg.tugraz.at}. Thomas
    Pock acknowledges support from the Austrian science fund (FWF)
    under the project EANOI, No. I1148 and the ERC starting grant
    HOMOVIS, No. 640156.} \and Shoham Sabach\footnote{Department of
    Industrial Engineering and Management, Technion---Israel Institute
    of Technology, Haifa 3200003, Israel. E-mail:
    ssabach@ie.technion.ac.il}}
\begin{document}
\maketitle

    \begin{abstract}
		In this paper we study nonconvex and nonsmooth optimization problems with semi-algebraic data, 
		where the variables vector is split into several blocks of variables. The problem consists of 
		one smooth function of the entire variables vector and the sum of nonsmooth functions for each 
		block separately. We analyze an inertial version of the Proximal Alternating Linearized 
		Minimization (PALM) algorithm and prove its global convergence to a critical point of the 
		objective function at hand. We illustrate our theoretical findings by presenting numerical 
		experiments on blind image deconvolution, on sparse non-negative matrix factorization and on 
		dictionary learning, which demonstrate the viability and effectiveness of the proposed method.
    \end{abstract}



\section{Introduction} \label{Sec:Intorduction}

	In the last decades advances in convex optimization have significantly influenced scientific fields 
	such as image processing and machine learning, which are dominated by computational approaches. 
	However, it is also known that the framework of convexity is often too restrictive to provide good 
	models for many practical problems. Several basic problems such as blind image deconvolution are 
	inherently nonconvex and hence there is a vital interest in the development of efficient and simple 
	algorithms for tackling nonconvex optimization problems.
\medskip

	A large part of the optimization community has also been devoted to the development of general 
	purpose solvers \cite{Nocedal06}, but in the age of big data, such algorithms often come to their 
	limits since they cannot efficiently exploit the structure of the problem at hand. One notable 
	exception is the general purpose limited quasi Newton method \cite{LiuNocedal89} which has been 
	published more than 25 years ago but still remains a competitive method.
\medskip

	A promising approach to tackle nonconvex problems is to consider a very rich the class of problems 
	which share certain structure that allows the development of efficient algorithms. One such class of 
	nonconvex optimization problems is given by the sum of three functions:
	\begin{equation} \label{eq:problem}
		\min_{\bx = \left(x_{1} , x_{2}\right)} F\left(\bx\right) := f_{1}\left(x_{1}\right) + f_{2}
		\left(x_{2}\right) + H\left(\bx\right),
	\end{equation}
	where $f_{1}$ and $f_{2}$ are assumed to be general nonsmooth and nonconvex functions with 
	efficiently computable proximal mappings (see exact definition in the next section) and $H$ is a 
	smooth coupling function which is required to have only partial Lipschitz continuous gradients 
	$\nabla_{x_{1}} H$ and $\nabla_{x_{2}} H$ (it should be noted that $\nabla_{\bx} H$ might not be a 
	Lipschitz continuous). 
\medskip

	Many practical problems frequently used in the machine learning and image processing communities 
	fall into this class of problems. Let us briefly mention two classical examples (another example 
	will be discussed in Section \ref{Sec:Numerics}).
\medskip

	The first example is Non-negative Matrix Factorization (NMF) \cite{PT1994,LS1999}. Given a 
	non-negative data matrix $A \in \rr_{+}^{m \times n}$ and an integer $r > 0$, the idea is to 
	approximate the matrix $A$ by a product of again non-negative matrices $BC$, where $B \in \rr_{+}^{m 
	\times r}$ and $C \in \rr_{+}^{r \times n}$. It should be noted that the dimension $r$ is usually 
	much smaller than $\min\{ m , n \}$. Clearly this problem is very difficult to solve and hence 
	several algorithms have been developed (see, for example, \cite{SD2006}). One possibility to solve 
	this problem is by finding a solution for the non-negative least squares model given by
	\begin{equation} \label{eq:NNMF}
		\min_{B , C} \tfrac12\norm{A - BC}_{F}^{2}, \quad \text{s.t.} \; B \geq 0, \; C \geq 0,
	\end{equation}  
	where the non-negativity constraint is understood pointwise and $\norm{\cdot}_{F}$ denotes the 
	classical Frobenius norm. The NMF has important applications in image processing (face recognition) 
	and bioinformatics (clustering of gene expressions). Observe that the gradient of the objective 
	function is not Lipschitz continuous but it is partially Lipschitz continuous which enables the 
	application of alternating minimization based methods (see \cite{BST2014}). Additionally, it is 
	popular to impose sparsity constraints on one or both of the unknowns, \eg $\norm{C}_{0} \leq c$, to 
	promote sparsity in the representation. See \cite{BST2014} for the first globally convergent 
	algorithm for solving the sparse NMF problem. As we will see, the complicated sparse NMF can be also 
	simply handled by our proposed algorithm which seems to produce better performances (see Section 
	\ref{Sec:Numerics}).
\medskip

	The second example we would like to mention is the important but ever challenging problem of blind 
	image deconvolution (BID) \cite{Levin2009}. Let $A \in \left[0 , 1\right]^{M \times N}$ be the 
	observed blurred image of size $M \times N$, and let $B \in \left[0 , 1\right]^{M \times N}$ be 
	the unknown sharp image of the same size. Furthermore, let $K \in \Delta_{mn}$ denote a small 
	unknown blur kernel (point spread function) of size $m \times n$, where $\Delta_{mn}$ denotes the 
	$mn$-dimensional standard unit simplex. We further assume that the observed blurred image has been
	formed by the following linear image formation model:
	\begin{equation*}			
		A = B \ast K + E,
	\end{equation*}
	where $\ast$ denotes a two dimensional discrete convolution operation and $E$ denotes a small 
	additive Gaussian noise. A typical variational formulation of the blind deconvolution problem is 
	given by:
	\begin{equation} \label{BID}
		\min_{U , K} \mathcal{R}\left(U\right) + \frac12\norm{B \ast K - A}_{F}^{2}, \quad \text{s.t.} 
		\; 0 \leq U \leq 1, \; K \in \Delta_{mn}.
	\end{equation}
 	In the above variational model, $\mathcal{R}$ is an image regularization term, typically a function, 
 	that imposes sparsity on the image gradient and hence favoring sharp images over blurred images.
\medskip

	We will come back to both examples in Section~\ref{Sec:Numerics} where we will show how the proposed 
	algorithm can be applied to efficiently solve these problems.
\medskip

	In \cite{BST2014}, the authors proposed a proximal alternating linearized minimization method (PALM) 
	that efficiently exploits the structure of problem \eqref{eq:problem}. PALM can be understood as a 
	blockwise application of the well-known proximal forward-backward algorithm \cite{LM79,CW05} in the 
	nonconvex setting. In the case that the objective function $F$ satisfy the so-called 
	Kurdyka-{\L}ojasiewicz (KL) property (the exact definition will be given in Section 
	\ref{Sec:Preli}), the whole sequence of the algorithm is guaranteed to converge to a critical point 
	of the problem.
\medskip

	In this paper, we propose an inertial version of the PALM algorithm and show convergence of the 
	whole sequence in case the objective function $F$ satisfy the KL property. The inertial term is 
	motivated from the Heavy Ball method of Polyak \cite{Polyak64} which in its most simple version 
	applied to minimizing a smooth function $f$ and can be written as the iterative scheme
	\begin{equation*}
		x^{k + 1} = x^{k} - \tau\nabla f\left(x^{k}\right) + \beta\left(x^{k} - x^{k - 1}\right),
	\end{equation*}
	where $\beta$ and $\tau$ are suitable parameters that ensure convergence of the algorithm. The heavy 
	ball method differs from the usual gradient method by the additional inertial term $\beta\left(x^{k} 
	- x^{k - 1}\right)$, which adds part of the old direction to the new direction of the algorithm. 
	Therefore for $\beta = 0$, we completely recover the classical algorithm of unconstrained 
	optimization, the Gradient Method. The heavy ball method can be motivated from basically three view 
	points.
\medskip

	First, the heavy ball method can be seen as an explicit finite differences discretization of the 
	heavy ball with friction dynamical system (see \cite{AlvarezAttouch2001}):
	\begin{equation*}
		\ddot x\left(t\right) + c\dot x\left(t\right) + g\left(x\left(t\right)\right) = 0,
	\end{equation*}
	where $x\left(t\right)$ is a time continuous trajectory, $\ddot x\left(t\right)$ is the 
	acceleration, $c \dot x\left(t\right)$ for $c > 0$ is the friction (damping), which is proportional 
	to the velocity $\dot x\left(t\right)$, and $g\left(x\left(t\right)\right)$ is an external 
	gravitational field. In the case that $g = \nabla f$ the trajectory $x\left(t\right)$ is running 
	down the ``energy landscape'' described by the objective function $f$ until a critical point 
	($\nabla f = 0$) is reached. Due to the presence of the inertial term, it can also overcome 
	spurious critical points of $f$, \eg saddle points.
\medskip

	Second, the heavy ball method can be seen as a special case of the so-called multi-step algorithms 
	where each step of the algorithm is given as a linear combination of all previously computed 
	gradients \cite{Drori2013}, that is, algorithm of the following form
	\begin{equation*}
		x^{k + 1} = x^{k} - \sum_{i = 0}^{k} \alpha_{i}\nabla f\left(x^{i}\right).
	\end{equation*}
	Let us note that in the case that the objective function $f$ is quadratic, the parameters 
	$\alpha_{i}$ can be chosen in a way such that the objective function is minimized at each step. This 
	approach eventually leads to the Conjugate Gradient (CG) method, pointing out a close relationship 
	to inertial based methods.
\medskip

	Third, accelerated gradient methods, as pioneered by Nesterov (see \cite{N04} for an overview), are 
	based on a variant of the heavy ball method that use the extrapolated point (based on the inertial 
	force) also for evaluating the gradient in the current step. It turns out that, in the convex 
	setting, these methods improve the worst convergence rate from $\mathcal{O}(1/k)$ to $\mathcal{O}(1/
	k^2)$, while leaving the computational complexity of each step basically the same.
\medskip

	In \cite{ZK93}, the heavy ball method has been analyzed for the first time in the setting of 
	nonconvex problems. It is shown that the heavy ball method is attracted by the connected components 
	of critical points. The proof is based on considering a suitable Lyapunov function that allows to 
	consider the two-step algorithm as a one-step algorithm. As we will see later, our convergence proof 
	is also based on rewriting the algorithm as a one-step method.
\medskip

	In \cite{OCBP2014}, the authors developed an inertial proximal gradient algorithm (iPiano). The 
	algorithm falls into the class of forward-backward splitting algorithms \cite{CW05}, as it performs 
	an explicit forward (steepest descent) step with respect to the smooth (nonconvex) function followed 
	by a (proximal) backward step with respect to the nonsmooth (convex) function. Motivated by the 
	heavy ball algorithm mentioned before, the iPiano algorithm makes use of an inertial force which 
	empirically shows to improve the convergence speed of the algorithm. A related method based on 
	general Bregman proximal-like distance functions has been recently proposed in \cite{Bot2015}.
\medskip

	Very recently, a randomized proximal linearization method has been proposed in \cite{Xu2014}. The 
	method is closely related to our proposed algorithm but convergence is proven only in the case that 
	the function values are strictly decreasing. This is true only if the inertial force is set to be 
	zero or the algorithm is restarted whenever the function values are not decreasing. In this paper, 
	however, we overcome this major drawback and prove convergence of the algorithm without any 
	assumption on the monotonicity of the objective function.
\medskip

	The remainder of the paper is organized as follows. In Section \ref{Sec:Problem} we give an exact 
	definition of the problem and the proposed algorithm. In Section \ref{Sec:Preli} we state few 
	technical results that will be necessary for the convergence analysis, which will be presented in 
	Section \ref{Sec:Convergence}. In Section \ref{Sec:Numerics} we present some numerical results and 
	analyze the practical performance of the algorithm in dependence of its inertial parameters.

\section{Problem Formulation and Algorithm} \label{Sec:Problem}
	In this paper we follow \cite{BST2014} and consider the broad class of nonconvex and nonsmooth
	problems of the following form
    \begin{equation} \label{Model}
        \mbox{minimize } F\left(\bx\right) := f_{1}\left(x_{1}\right) + f_{2}\left(x_{2}\right) + H
        \left(\bx\right) \mbox{ over all } \bx = \left(x_{1} , x_{2}\right) \in \rr^{n_{1}} \times
        \rr^{n_{2}},
    \end{equation}
    where $f_{1}$ and $f_{2}$ are extended valued (\ie giving the possibility of imposing constraints 
    separately on the blocks $x_{1}$ and $x_{2}$) and $H$ is a smooth coupling function (see below for 
    more precise assumptions on the involved functions). We would like to stress from the beginning that 
    \textit{even though} all the discussions and results of this paper derived for two blocks of 
    variables $x_{1}$ and $x_{2}$, they hold true for any finite number of blocks. This choice was done 
    only for the sake of simplicity of the presentation of the algorithm and the convergence results.
\medskip

	As we discussed in the introduction, the proposed algorithm can be viewed either as a block version 
	of the recent iPiano algorithm \cite{OCBP2014} or as an inertial based version of the recent PALM 
	algorithm \cite{BST2014}. Before presenting the algorithm it will be convenient to recall the 
	definition of the Moreau proximal mapping \cite{M65}. Given a proper and lower semicontinuous 
	function $\s : \rr^{d} \rightarrow \left(-\infty , \infty\right]$, the proximal mapping associated 
	with $\s$ is defined by
	\begin{equation} \label{D:ProximalMap}
        \prox_{t}^{\s}\left(p\right) := \argmin \left\{ \s\left(q\right) + \frac{t}{2}\norm{q - p}^{2} : 
        \; q \in \rr^{d} \right\}, \quad \left(t > 0\right).
    \end{equation}
	Following \cite{BST2014}, we take the following as our blanket assumption.
    \begin{assumption} \label{AssumptionsA}
        \begin{itemize}
            \item[$\rm{(i)}$] $f_{1} : \rr^{n_{1}} \rightarrow \left(-\infty , \infty\right]$ and $f_{2} 
            		: \rr^{n_{2}} \rightarrow \left(-\infty , \infty\right]$ are proper and lower 
            		semicontinuous functions such that $\inf_{\rr^{n_{1}}} f_{1} > -\infty$ and 
            		$\inf_{\rr^{n_{2}}} f_{2} > -\infty$.
            \item[$\rm{(ii)}$] $H : \rr^{n_{1}} \times \rr^{n_{2}} \rightarrow \rr$ is differentiable 
            		and $\inf_{\rr^{n_{1}} \times \rr^{n_{2}}} F > -\infty$.
            \item[$\rm{(iii)}$] For any fixed $x_{2}$ the function $x_{1} \rightarrow H\left(x_{1} , 
	            x_{2}\right)$ is $C^{1,1}_{L_{1}(x_{2})}$, namely the partial gradient $\nabla_{x_{1}} H
	            \left(x_{1} , x_{2}\right)$ is globally Lipschitz with moduli $L_{1}\left(x_{2}\right)$, 
	            that is,
                \begin{equation*}
                    \norm{\nabla_{x_{1}} H\left(u , x_{2}\right) - \nabla_{x_{1}} H\left(v , x_{2}
                    \right)} \leq L_{1}\left(x_{2}\right)\norm{u - v}, \quad \forall \; u , v \in 
                    \rr^{n_{1}}.
                \end{equation*}
				Likewise, for any fixed $x_{1}$ the function $x_{2} \rightarrow H\left(x_{1} , x_{2}
                	\right)$ is assumed to be $C^{1,1}_{L_{2}(x_{1})}$.
            \item[$\rm{(iv)}$] For $i = 1 , 2$ there exists $\lambda_{i}^{-} , \lambda_{i}^{+} > 0$ 
            		such that
                	\begin{align}
                		\inf \left\{ L_{1}\left(x_{2}\right) : x_{2} \in B_{2} \right\} & \geq \lambda_{1}
                		^{-} \quad \text{and} \quad \inf \left\{ L_{2}\left(x_{1}\right) : x_{1} \in B_{1} 
                		\right\} \geq \lambda_{2}^{-}, \label{A:InfBound} \\
                    \sup \left\{ L_{1}\left(x_{2}\right) : x_{2} \in B_{2} \right\} & \leq \lambda_{1}
                    ^{+} \quad \text{and} \quad \sup \left\{ L_{2}\left(x_{1}\right) : x_{1} \in B_{1} 
                    \right\} \leq \lambda_{2}^{+}, \label{A:SupBound}
                \end{align}
                for any compact set $B_{i} \subseteq \rr^{n_{i}}$, $i = 1 , 2$.
            \item[$\rm{(v)}$] $\nabla H$ is Lipschitz continuous on bounded subsets of $\rr^{n_{1}} 
            		\times \rr^{n_{2}}$. In other words, for each bounded subset $B_{1} \times B_{2}$ of 
            		$\rr^{n_{1}} \times \rr^{n_{2}}$ there exists $M > 0$ such that:
                \begin{align*}
                    \norm{\left(\nabla_{x_{1}} H\left(x_{1} , x_{2}\right) - \nabla_{x_{1}} H\left(y_{1} 
                    , y_{2}\right) , \nabla_{x_{2}} H\left(x_{1} , x_{2}\right) - \nabla_{x_{2}} H
                    \left(y_{1} , y_{2}\right)\right)} & \nonumber \\
                    & \hspace{-1in} \leq M \norm{\left(x_{1} - y_{1} , x_{2} - y_{2}\right)}. 
                \end{align*}
        \end{itemize}
    \end{assumption}
    We propose now the inertial Proximal Alternating Linearized Minimization (iPALM) algorithm.
\vspace{0.2in}

	{\center\fbox{\parbox{16cm}{{\bf iPALM: Inertial Proximal Alternating Linearized Minimization}
		\begin{itemize}
        		\item[1.] Initialization: start with any $\left(x_{1}^{0} , x_{2}^{0}\right) \in \rr^{n_{1}} 
            		\times \rr^{n_{2}}$.
            	\item[2.] For each $k = 1 , 2 , \ldots$ generate a sequence $\left\{ \left(x_{1}^{k} , x_{2}
            		^{k}\right) \right\}_{k \in \nn}$ as follows:
            		\begin{itemize}
	               	\item[2.1.] Take $\alpha_{1}^{k} , \beta_{1}^{k} \in \left[0 , 1\right]$ and 
	                		$\tau_{1}^{k} > 0$. Compute
	                    	\begin{align} 
    			                	y_{1}^{k} & = x_{1}^{k} + \alpha_{1}^{k}\left(x_{1}^{k} - x_{1}^{k - 1}
    			                	\right), \label{iPALM:StepX1:1} \\
		                    	z_{1}^{k} & = x_{1}^{k} + \beta_{1}^{k}\left(x_{1}^{k} - x_{1}^{k - 1}
		                    	\right), \label{iPALM:StepX1:2} \\
	                        x_{1}^{k + 1} & \in \prox_{\tau_{1}^{k}}^{f_{1}}\left(y_{1}^{k} - \frac{1}
	                        {\tau_{1}^{k}}\nabla_{x_{1}} H\left(z_{1}^{k} , x_{2}^{k}\right)\right). 
	                        \label{iPALM:StepX1:3} 
    		                \end{align}
	                	\item[2.2.] Take $\alpha_{2}^{k} , \beta_{2}^{k} \in \left[0 , 1\right]$ and 
	                		$\tau_{2}^{k} > 0$. Compute
                		    	\begin{align} 
                    			y_{2}^{k} & = x_{2}^{k} + \alpha_{2}^{k}\left(x_{2}^{k} - x_{2}^{k - 1}
                    			\right), \label{iPALM:StepX2:1} \\
		                    	z_{2}^{k} & = x_{2}^{k} + \beta_{2}^{k}\left(x_{2}^{k} - x_{2}^{k	- 1}
		                    	\right), \label{iPALM:StepX2:2} \\
                    		    x_{2}^{k + 1} & \in \prox_{\tau_{2}^{k}}^{f_{2}}\left(y_{2}^{k} - \frac{1}
                    		    {\tau_{2}^{k}}\nabla_{x_{2}} H\left(x_{1}^{k + 1} , z_{2}^{k}\right)
                    		    \right). \label{iPALM:StepX2:3} 
                    \end{align}
        			\end{itemize}
        \end{itemize}}}}
\vspace{0.2in}

	The parameters $\tau_{1}^{k}$ and $\tau_{2}^{k}$, $k \in \nn$, are discussed in Section
	\ref{Sec:Convergence} but for now, we can say that they are proportional to the respective 
	partial Lipschitz moduli of $H$. The larger the partial Lipschitz moduli the smaller the step-size, 
	and hence the slower the algorithm. As we shall see below, the partial Lipschitz moduli $L_{1}
	\left(x_{2}\right)$ and $L_{2}\left(x_{1}\right)$ are explicitly available for the examples 
	mentioned at the introduction. However, note that if these are unknown, \textit{or 
	still too difficult to compute}, then a backtracking scheme \cite{BT09} can be incorporated and the 
	convergence results developed below remain true, for simplicity of exposition we omit the details.
\medskip	
	
	In Section \ref{Sec:Numerics}, we will show that the involved functions of the non-negative matrix 
	factorization model (see \eqref{eq:NNMF}) and of the blind image deconvulation model (see 
	\eqref{BID}) do satisfy Assumption \ref{AssumptionsA}. For the general setting we point out the 
	following remarks about Assumption \ref{AssumptionsA}.
	\begin{itemize}
		\item[$\rm{(i)}$] The first item of Assumption \ref{AssumptionsA} is very general and most of 
			the interesting constraints (via their indicator functions) or regularizing functions 
			fulfill these requirements. 
        	\item[$\rm{(ii)}$] Items (ii)-(v) of Assumption \ref{AssumptionsA} are beneficially exploited to 
        		build the proposed iPALM algorithm. These requirements do not guarantee that the gradient of 
        		$H$ is globally Lipschitz (which is the case in our mentioned applications). The fact that 
        		$\nabla H$ is not globally Lipschitz reduces the potential of applying the iPiano and PFB 
        		methods in concrete applications and therefore highly motivated us to study their block 
        		counterparts (PALM in \cite{BST2014} and iPALM in this paper).
        	\item[$\rm{(iii)}$] Another advantage of algorithms that exploit block structures inherent in 
        		the model at hand is the fact that they achieve better numerical performance (see Section 
        		\ref{Sec:Numerics}) by taking step-sizes which is optimized to each separated block of 
        		variables.
        	\item[$\rm{(iv)}$] Item (v) of Assumption \ref{AssumptionsA} holds true, for example, when $H$ 
        		is $C^{2}$. In this case the inequalities in \eqref{A:SupBound} could be obtained if the 
        		sequence, which generated by the algorithm, is bounded.
	\end{itemize}
\medskip

	The iPALM algorithm generalizes few known algorithms for different values of the inertial parameters 
	$\alpha_{i}^{k}$ and $\beta_{i}^{k}$, $k \in \nn$ and $i = 1  ,2$. For example, when $\alpha_{i}^{k} 
	= \beta_{i}^{k} = 0$, $k \in \nn$, we recover the PALM algorithm of \cite{BST2014} which is a block 
	version of the classical Proximal Forward-Backward (PFB) algorithm. When, there is only one block of 
	variables, for instance only $i = 1$, we get the iPiano algorithm \cite{OCBP2014} which is recovered 
	exactly only when $\beta_{1}^{k} = 0$, $k \in \nn$. It should be also noted that in \cite{OCBP2014}, 
	the authors additionally assume that the function $f_{1}$ is convex (an assumption that is not 
	needed in our case). The iPiano algorithm by itself generalizes two classical and known algorithms, 
	one is the Heavy-Ball method \cite{B1964} (when $f_{1} \equiv 0$) and again the PFB method (when 
	$\alpha_{1}^{k} = 0$, $k \in \nn$).
	
\section{Mathematical Preliminaries and Proof Methodology} \label{Sec:Preli}	
	Throughout this paper we are using standard notations and definitions of nonsmooth analysis which 
	can be found in any classical book, see for instance \cite{RW1998-B,M2006-B}. We recall here few 
	notations and technical results. Let $\s : \rr^{d} \rightarrow \left(-\infty , \infty\right]$ be a 
	proper and lower semicontinuous function. Since we are dealing with nonconvex and nonsmooth 
	functions that can have the value $\infty$, we use the notion of limiting subdifferential (or simply 
	subdifferential), see \cite{M2006-B}, which is denoted by $\partial \s$. In what follows, we are 
	interested in finding critical points of the objective function $F$ defined in \eqref{Model}. 
	Critical points are those points for which the corresponding subdifferential contains the zero 
	vector $\bo$. The set of critical points of $\s$ is denoted by $\crit{\s}$, that is,
	\begin{equation*}
		\crit{\s} = \left\{ u \in \dom{\s} : \, \bo \in \partial \s\left(u\right) \right\}.
	\end{equation*}
	An important property of the subdifferential is recorded in the following remark (see 
	\cite{RW1998-B}).
	\begin{remark} \label{R:SubdiffClosed}
		Let $\left\{ \left(u^{k} , q^{k}\right) \right\}_{k \in \nn}$ be a sequence in 
		$\gr{\left(\partial \s\right)}$ that converges to $\left(u , q\right)$ as $k \rightarrow \infty
		$. By the definition of $\partial \s\left(u\right)$, if $\s\left(u^{k}\right)$ converges to 
		$\s\left(u\right)$ as $k \rightarrow \infty$, then $\left(u , q\right) \in \gr{(\partial \s)}$.
	\end{remark}
	The convergence analysis of iPALM is based on the proof methodology which was developed in 
	\cite{ABS2013} and more recently extended and simplified in \cite{BST2014}. The main part of the 
	suggested methodology relies on the fact that the objective function of the problem at hand 
	satisfies the Kurdyka-{\L}ojasiewicz (KL) property. Before stating the KL property we will need the 
	definition of the following class of desingularizing functions. For $\eta \in \left(0 , \infty
	\right]$ define
	\begin{equation}
   		\Phi_{\eta} \equiv \left\{ \varphi \in C\left[\left[0 , \eta\right) , \rr_{+}\right] \mbox{ such 
   		that } \left.
   			\begin{cases}
			   	\varphi\left(0\right) = 0 & \\ 
			   	\varphi \in C^{1} & \mbox{ on } \left(0 ,\eta\right) \\
				\varphi'\left(s\right) > 0 & \mbox{ for all } s \in \left(0 , \eta\right)
			\end{cases} \right\} \right\}.
	\end{equation}
	The function $\s$ is said to have the Kurdyka-{\L}ojasiewicz (KL) property at $\overline{u} \in \dom 
	\partial \s$ if there exist $\eta \in \left(0 , \infty\right]$, a neighborhood $U$ of $\overline{u}$ 
	and a function $\varphi \in \Phi_{\eta}$, such that, for all 
	\begin{equation*}
		u \in U \cap [\s(\overline{u}) < \s(u) < \s(\overline{u}) + \eta],
	\end{equation*}
	the following inequality holds
	\begin{equation}\label{e:KL}
   		\varphi\left(\s\left(u\right) - \s\left(\overline{u}\right)\right)\dist\left(0 , \partial \s
   		\left(u\right)\right) \geq 1,
	\end{equation}
	where for any subset $S \subset \rr^{d}$ and any point $x \in \rr^{d}$
    \begin{equation*}
        \dist\left(x , S\right) := \inf \left\{ \norm{y - x} : \; y \in S \right\}.
    \end{equation*}
    When $S = \emptyset$, we have that $\dist\left(x , S\right) = \infty$ for all $x$. If $\s$ satisfies 
    property \eqref{e:KL} at each point of $\dom{\partial \s}$, then $\s$ is called a {\em KL function}.
\medskip

	The convergence analysis presented in the following section is based on the uniformized KL property 
	which was established in \cite[Lemma 6, p. 478]{BST2014}.
    \begin{lemma} \label{L:KLProperty}
        Let $\Omega$ be a compact set and let $\s : \rr^{d} \rightarrow \left(-\infty , \infty\right]$ 
        be a proper and lower semicontinuous function. Assume that $\s$ is constant on $\Omega$ and 
        satisfies the KL property at each point of $\Omega$. Then, there exist $\varepsilon > 0$, $\eta 
        > 0$ and $\varphi \in \Phi_{\eta}$ such that for all $\overline{u}$ in $\Omega$ and all $u$ in 
        the following intersection
        \begin{equation} \label{L:KLProperty:1}
            \left\{ u \in \rr^{d} : \; \dist\left(u , \Omega\right) < \varepsilon \right\} \cap \left[ 
            \s\left(\overline{u}\right) < \s\left(u\right) < \s\left(\overline{u}\right) + \eta\right],
        \end{equation}
        one has,
        \begin{equation} \label{L:KLProperty:2}
            \varphi'\left(\s\left(u\right) - \s\left(\overline{u}\right)\right)\dist\left(0 , \partial 
            \s\left(u\right)\right) \geq 1.
        \end{equation}
    \end{lemma}
    We refer the reader to \cite{BDLM10} for a depth study of the class of KL functions. For the 
    important relation between semi-algebraic and KL functions see \cite{BDL2006}. In 
    \cite{AB2009,ABRS2010,ABS2013,BST2014}, the interested reader can find through catalog of functions 
    which are very common in many applications and satisfy the KL property.
\medskip

	Before concluding the mathematical preliminaries part we would like to mention few important 
	properties of the proximal map (defined in \eqref{D:ProximalMap}). The following result can be found 
	in \cite{RW1998-B}.
    \begin{proposition} \label{P:WellProximal}
        Let $\s : \rr^{d} \rightarrow \left(-\infty , \infty\right]$ be a proper and lower 
        semicontinuous function with $\inf_{\rr^{d}} \s > -\infty$. Then, for every $t \in \left(0 , 
        \infty\right)$ the set $\prox_{t\s}\left(u\right)$ is nonempty and compact.
    \end{proposition}
    It follows immediately from the definition that $\prox{\s}$ is a multi-valued map when $\s$ is 
    nonconvex. The multi-valued projection onto a nonempty and closed set $C$ is recovered when $\s = 
    \delta_{C}$, which is the indicator function of $C$ that defined to be zero on $C$ and $\infty$ 
    outside.
\medskip

	The main computational effort of iPALM involves a proximal mapping step of a proper and lower 
	semicontinuous but nonconvex function. The following property will be essential in the forthcoming 
	convergence analysis and is a slight modification of \cite[Lemma 2, p. 471]{BST2014}.
	\begin{lemma}[Proximal inequality] \label{L:ProxIne}
        Let $h : \rr^{d} \rightarrow \rr$ be a continuously differentiable function with gradient 
        $\nabla h$ assumed $L_{h}$-Lipschitz continuous and let $\s : \rr^{d} \rightarrow \left(-\infty 
        , \infty\right]$ be a proper and lower semicontinuous function with $\inf_{\rr^{d}} \s > - 
        \infty$. Then, for any $v , w \in \dom{\s}$ and any $u^{+} \in \rr^{d}$ defined by
        \begin{equation} \label{L:ProxIne:1}
            u^{+} \in \prox_{t}^{\s}\left(v - \frac{1}{t}\nabla h\left(w\right)\right), \quad t > 0,
        \end{equation}
        we have, for any $u \in \dom{\s}$ and any $s > 0$:
       	\begin{equation} \label{L:ProxIne:2}
            g\left(u^{+}\right) \leq g\left(u\right) + \frac{L_{h} + s}{2}\norm{u^{+} - u}^{2} + 
            \frac{t}{2}\norm{u - v}^{2} - \frac{t}{2}\norm{u^{+} - v}^{2} + \frac{L_{h}^{2}}{2s}\norm{u 
            - w}^{2},
        \end{equation}
        where $g := h + \s$.
    \end{lemma}
    \begin{proof}
        First, it follows immediately from Proposition \ref{P:WellProximal} that $u^{+}$ is  
        well-defined. By the definition of the proximal mapping (see \eqref{D:ProximalMap}) we get that
        \begin{equation*}
            u^{+} \in \argmin_{\xi  \in \rr^{d}} \left\{ \act{\xi - v , \nabla h\left(w\right)} + 
            \frac{t}{2}\norm{\xi - v}^{2} + \s\left(\xi\right) \right\},
        \end{equation*}
        and hence in particular, by taking $\xi = u$, we obtain
        \begin{equation*}
            \act{u^{+} - v , \nabla h\left(w\right)} + \frac{t}{2}\norm{u^{+} - v}^{2} + \s\left(u^{+}
            \right) \leq \act{u - v , \nabla h\left(w\right)} + \frac{t}{2}\norm{u - v}^{2} + \s\left(u
            \right).
        \end{equation*}
        Thus
        \begin{equation} \label{L:ProxIne:3}
            \s\left(u^{+}\right) \leq \act{u - u^{+} , \nabla h\left(w\right)} + \frac{t}{2}\norm{u - v}
            ^{2} - \frac{t}{2}\norm{u^{+} - v}^{2} + \s\left(u\right).
        \end{equation}
        Invoking first the descent lemma (see \cite{BT89}) for $h$, and using \eqref{L:ProxIne:3}, 
        yields
        \begin{align*}
            h\left(u^{+}\right) + \s\left(u^{+}\right) & \leq h\left(u\right) + \act{u^{+} - u , \nabla 
            h\left(u\right)} + \frac{L_{h}}{2}\norm{u^{+} - u}^{2} + \act{u - u^{+} , \nabla h\left(w
            \right)} \\
            & + \frac{t}{2}\norm{u - v}^{2} - \frac{t}{2}\norm{u^{+} - v}^{2} + \s\left(u\right) \\
            & = h\left(u\right) + \s\left(u\right) + \act{u^{+} - u , \nabla h\left(u\right) - \nabla 
            h\left(w\right)} + \frac{L_{h}}{2}\norm{u^{+} - u}^{2} \\
            & + \frac{t}{2}\norm{u - v}^{2} - \frac{t}{2}\norm{u^{+} - v}^{2}.
        \end{align*}
        Now, using the fact that $\act{p , q} \leq \left(s/2\right)\norm{p}^{2} + \left(1/2s\right)
        \norm{q}^{2}$ for any two vectors $p , q \in \rr^{d}$ and every $s > 0$, yields
        \begin{align*}
            \act{u^{+} - u , \nabla h\left(u\right) - \nabla h\left(w\right)} & \leq \frac{s}{2}
            \norm{u^{+} - u}^{2} + \frac{1}{2s}\norm{\nabla h\left(u\right) - \nabla h\left(w\right)}
            ^{2} \\
            & \leq \frac{s}{2}\norm{u^{+} - u}^{2} + \frac{L_{h}{^2}}{2s}\norm{u - w}^{2},
        \end{align*}
        where we have used the fact that $\nabla h$ is $L_{h}$-Lipschitz continuous. Thus, combining the 
        last two inequalities proves that \eqref{L:ProxIne:2} holds.
    \end{proof}
	\begin{remark}
		It should be noted that if the nonsmooth function $\s$ is also known to be convex, then we can 
		derive the following tighter upper bound (\cf \eqref{L:ProxIne:2})
      	\begin{equation} \label{R:ProxIneConvex}
            \hspace{-0.05in} g\left(u^{+}\right) \leq g\left(u\right) + \frac{L_{h} + s - t}{2}
            \norm{u^{+} - u}^{2} + \frac{t}{2}\norm{u - v}^{2} - \frac{t}{2}\norm{u^{+} - v}^{2} + 
            \frac{L_{h}^{2}}{2s}\norm{u - w}^{2}.
        	\end{equation}
	\end{remark}
	
\subsection{Convergence Proof Methodology}
	In this section we briefly summarize (\cf Theorem \ref{T:GeneralConvergence} below) the methodology 
	recently proposed in \cite{BST2014} which provides the key elements to obtain an abstract 
	convergence result that can be applied to any algorithm and will be applied here to prove 
	convergence of iPALM. Let $\Seq{\bu}{k}$ be a sequence in $\rr^{d}$ which was generated from a
	starting point $\bu^{0}$ by a generic algorithm ${\cal A}$. The set of all limit points of 
	$\Seq{\bu}{k}$ is denoted by $\omega\left(\bu^{0}\right)$, and defined by
    \begin{equation*}
    		\left\{ \overline{\bu} \in \rr^{d}: \; \exists \mbox{ an increasing sequence of integers } 
    		\seq{k}{l} \mbox{ such that }\; \bu^{k_{l}} \rightarrow \overline{\bu} \mbox{ as } l \rightarrow 
    		\infty \right\}.
    \end{equation*}	
	\begin{theorem} \label{T:GeneralConvergence}
		Let $\Psi : \rr^{d} \rightarrow \left(-\infty , \infty\right]$ be a proper, lower semicontinuous 
		and semi-algebraic function with $\inf \Psi > -\infty$. Assume that $\Seq{\bu}{k}$ is a bounded 
		sequence generated by a generic algorithm ${\cal A}$ from a starting point $\bu^{0}$, for which 
		the following three conditions hold true for any $k \in \nn$.
		\begin{itemize}
        		\item[$\rm{(C1)}$] There exists a positive scalar $\rho_{1}$ such that
            		\begin{equation*}
                		\rho_{1}\norm{\bu^{k + 1} - \bu^{k}}^{2} \leq \Psi\left(\bu^{k}\right) - \Psi
                		\left(\bu^{k + 1}\right), \quad \forall \,\, k = 0 , 1 , \ldots.
            		\end{equation*}
        		\item[$\rm{(C2)}$] There exists a positive scalar $\rho_{2}$ such that for some $\bw^{k} \in 
        			\partial \Psi\left(\bu^{k}\right)$ we have
            		\begin{equation*}
            		    \norm{\bw^{k}} \leq \rho_{2}\norm{\bu^{k} - \bu^{k - 1}}, \quad \forall \,\, k = 0 , 
            		    1 , \ldots.
            		\end{equation*}
    			\item[$\rm{(C3)}$] Each limit point in the set $\omega\left(\bu^{0}\right)$ is a critical 
    				point of $\Psi$, that is, $\omega\left(\bu^{0}\right) \subset \crit{\Psi}$.
        \end{itemize}	
		Then, the sequence $\Seq{\bu}{k}$ converges to a critical point $\bu^{\ast}$ of $\Psi$.    	
	\end{theorem}	    
	
\section{Convergence Analysis of iPALM} \label{Sec:Convergence}
	Our aim in this section is to prove that the sequence $\left\{ \left(x_{1}^{k} , x_{2}^{k}\right) 
	\right\}_{k \in \nn}$ which is generated by iPALM converges to a critical point of the objective 
	function $F$ defined in \eqref{Model}. To this end we will follow the proof methodology described 
	above in Theorem \ref{T:GeneralConvergence}. In the case of iPALM, similarly to the iPiano algorithm 
	(see \cite{OCBP2014}), it is not possible to prove that condition (C1) hold true for the sequence 
	$\left\{ \left(x_{1}^{k} , x_{2}^{k}\right) \right\}_{k \in \nn}$ and the function $F$, namely, this 
	is not a descent algorithm with respect to $F$. Therefore, we first show that conditions (C1), (C2) 
	and (C3) hold true for an auxiliary sequence and auxiliary function (see details below). Then, based 
	on these properties we will show that the original sequence converges to a critical point of the 
	original function $F$.
\medskip

	We first introduce the following notations that simplify the coming expositions. For any $k \in \nn
	$, we define
	\begin{equation} \label{D:Delta}
		\Delta_{1}^{k} = \frac{1}{2}\norm{x_{1}^{k} - x_{1}^{k - 1}}^{2}, \quad \Delta_{2}^{k} = 
		\frac{1}{2}\norm{x_{2}^{k} - x_{2}^{k - 1}}^{2} \quad \text{and} \quad \Delta^{k} = \frac{1}{2}
		\norm{\bx^{k} - \bx^{k - 1}}^{2},
	\end{equation}
	it is clear, that using these notations, we have that $\Delta^{k} = \Delta_{1}^{k} + \Delta_{2}^{k}$ 
	for all $k \in \nn$. Using these notations we can easily show few basic relations of the sequences 
	$\left\{ x_{i}^{k} \right\}_{k \in \nn}$, $\left\{ y_{i}^{k} \right\}_{k \in \nn}$, and $\left\{ 
	z_{i}^{k} \right\}_{k \in \nn}$, for $i = 1 , 2$, generated by iPALM.
	\begin{proposition} \label{P:BasicFacts}
		Let $\left\{ \left(x_{1}^{k} , x_{2}^{k}\right) \right\}_{k \in \nn}$ be a sequence generated by 
		iPALM. Then, for any $k \in \nn$ and $i = 1 , 2$, we have
		\begin{itemize}
        		\item[$\rm{(i)}$] $\norm{x_{i}^{k} - y_{i}^{k}}^{2} = 2\left(\alpha_{i}^{k}\right)^{2}
        			\Delta_{i}^{k}$;
        		\item[$\rm{(ii)}$] $\norm{x_{i}^{k} - z_{i}^{k}}^{2} = 2\left(\beta_{i}^{k}\right)^{2}
        			\Delta_{i}^{k}$;
    			\item[$\rm{(iii)}$] $\norm{x_{i}^{k + 1} - y_{i}^{k}}^{2} \geq 2\left(1 - \alpha_{i}^{k}
    				\right)\Delta_{i}^{k + 1} + 2\alpha_{i}^{k}\left(\alpha_{i}^{k} - 1\right)\Delta_{i}^{k}
    				$.
        \end{itemize}	
	\end{proposition}
	\begin{proof}
		The first two items follow immediately from the facts that $x_{i}^{k} - y_{i}^{k} = \alpha_{i}
		^{k}\left(x_{i}^{k - 1} - x_{i}^{k}\right)$ and $x_{i}^{k} - z_{i}^{k} = \beta_{i}^{k}
		\left(x_{i}^{k - 1} - x_{i}^{k}\right)$, for $i = 1 , 2$ (see steps \eqref{iPALM:StepX1:1}, 
		\eqref{iPALM:StepX1:2}, \eqref{iPALM:StepX2:1} and \eqref{iPALM:StepX2:2}). The last item 
		follows from the following argument
		\begin{align}
	       \norm{x_{i}^{k + 1} - y_{i}^{k}}^{2} & = \norm{x_{i}^{k + 1} - x_{i}^{k} - \alpha_{i}^{k}
	       \left(x_{i}^{k} - x_{i}^{k - 1}\right)}^{2} \nonumber \\
	       & = 2\Delta_{i}^{k + 1} - 2\alpha_{i}^{k}\act{x_{i}^{k + 1} - x_{i}^{k} , x_{i}^{k} - x_{i}
	       ^{k - 1}} + 2\left(\alpha_{i}^{k}\right)^{2}\Delta_{i}^{k} \nonumber \\
	       & \geq 2\left(1 - \alpha_{i}^{k}\right)\Delta_{i}^{k + 1} + 2\alpha_{i}^{k}\left(\alpha_{i}
	       ^{k} - 1\right)\Delta_{i}^{k}, \label{L:Dec:6}
        \end{align}
		where we have used the fact that
		\begin{equation*}
			2\act{x_{i}^{k + 1} - x_{i}^{k} , x_{i}^{k} - x_{i}^{k - 1}} \leq \norm{x_{i}^{k + 1} - 
			x_{i}^{k}}^{2} + \norm{x_{i}^{k} - x_{i}^{k - 1}}^{2} = 2\Delta_{i}^{k + 1} + 2\Delta_{i}
			^{k},
		\end{equation*}
		that follows from the Cauchy-Schwartz and Young inequalities. This proves item (iii).
	\end{proof}
	Now we prove the following property of the sequence $\left\{ \left(x_{1}^{k} , x_{2}^{k}\right) 
	\right\}_{k \in \nn}$ generated by iPALM.
	\begin{proposition} \label{TechC1}
		Suppose that Assumption \ref{AssumptionsA} holds. Let $\left\{ \left(x_{1}^{k} , x_{2}^{k}
		\right) \right\}_{k \in \nn}$ be a sequence generated by iPALM, then for all $k \in \nn$, we 
		have that
		\begin{align*}
        		F\left(\bx^{k + 1}\right) & \leq F\left(\bx^{k}\right) + \frac{1}{s_{1}^{k}}\left(L_{1}
        		(x_{2}^{k})^{2}\left(\beta_{1}^{k}\right)^{2} + s_{1}^{k}\tau_{1}^{k}\alpha_{1}^{k}\right)
        		\Delta_{1}^{k} + \frac{1}{s_{2}^{k}}\left(L_{2}(x_{1}^{k + 1})^{2}\left(\beta_{2}^{k}
        		\right)^{2} + s_{2}^{k}\tau_{2}^{k}\alpha_{2}^{k}\right)\Delta_{2}^{k} \\
        		& + \left(L_{1}(x_{2}^{k}) + s_{1}^{k} - \tau_{1}^{k}\left(1 - \alpha_{1}^{k}\right)\right)
        		\Delta_{1}^{k + 1} + \left(L_{2}(x_{1}^{k + 1}) + s_{2}^{k} - \tau_{2}^{k}\left(1 - 
        		\alpha_{2}^{k}\right)\right)\Delta_{2}^{k + 1},
        \end{align*}	
        where $s_{1}^{k} > 0$ and $s_{2}^{k} > 0$ are arbitrarily chosen, for all $k \in \nn$.
	\end{proposition}
	\begin{proof}
		Fix $k \geq 1$. Under our Assumption \ref{AssumptionsA}(ii), the function $x_{1} \rightarrow H
		\left(x_{1} , x_{2}\right)$ ($x_{2}$ is fixed) is differentiable and has a Lipschitz continuous 
		gradient with moduli $L_{1}\left(x_{2}\right)$. Using the iterative step \eqref{iPALM:StepX1:3}, 
		applying Lemma \ref{L:ProxIne} for $h\left(\cdot\right) := H\left(\cdot , x_{2}^{k}\right)$, $\s 
		:= f_{1}$ and $t := \tau_{1}^{k}$ with the points $u = x_{1}^{k}$, $u^{+} = x_{1}^{k + 1}$, $v = 
		y_{1}^{k}$ and $w = z_{1}^{k}$ yields that
		\begin{align}
        		H\left(x_{1}^{k + 1} , x_{2}^{k}\right) + f_{1}\left(x_{1}^{k + 1}\right) & \leq 
        		H\left(x_{1}^{k} , x_{2}^{k}\right) + f_{1}\left(x_{1}^{k}\right) + \frac{L_{1}(x_{2}^{k}) + 
        		s_{1}^{k}}{2}\norm{x_{1}^{k + 1} - x_{1}^{k}}^{2} \nonumber \\
        		& + \frac{\tau_{1}^{k}}{2}\norm{x_{1}^{k} - y_{1}^{k}}^{2} - \frac{\tau_{1}^{k}}{2}
        		\norm{x_{1}^{k + 1} - y_{1}^{k}}^{2} + \frac{L_{1}(x_{2}^{k})^{2}}{2s_{1}^{k}}\norm{x_{1}
        		^{k} - z_{1}^{k}}^{2} \nonumber \\
        		& \leq H\left(x_{1}^{k} , x_{2}^{k}\right) + f_{1}\left(x_{1}^{k}\right) + \left(L_{1}(x_{2}
        		^{k}) + s_{1}^{k}\right)\Delta_{1}^{k + 1} + \tau_{1}^{k}\left(\alpha_{1}^{k}\right)^{2}
        		\Delta_{1}^{k} \nonumber \\
        		& - \tau_{1}^{k}\left(\left(1 - \alpha_{1}^{k}\right)\Delta_{1}^{k + 1} + \alpha_{1}^{k}
        		\left(\alpha_{1}^{k} - 1\right)\Delta_{1}^{k}\right) + \frac{L_{1}(x_{2}^{k})^{2}
        		\left(\beta_{1}^{k}\right)^{2}}{s_{1}^{k}}\Delta_{1}^{k} \nonumber \\
        		& = H\left(x_{1}^{k} , x_{2}^{k}\right) + f_{1}\left(x_{1}^{k}\right) + \left(L_{1}
        		(x_{2}^{k}) + s_{1}^{k} - \tau_{1}^{k}\left(1 - \alpha_{1}^{k}\right)\right)\Delta_{1}^{k + 
        		1} \nonumber \\
        		& + \frac{1}{s_{1}^{k}}\left(L_{1}(x_{2}^{k})^{2}\left(\beta_{1}^{k}\right)^{2} + s_{1}^{k}
        		\tau_{1}^{k}\alpha_{1}^{k}\right)\Delta_{1}^{k}, \label{TechC1:1}
        	\end{align}
        where the second inequality follows from Proposition \ref{P:BasicFacts}. Repeating all the
        arguments above on the iterative step \eqref{iPALM:StepX2:3} yields the following
		\begin{align}
        		H\left(x_{1}^{k + 1} , x_{2}^{k + 1}\right) + f_{2}\left(x_{2}^{k + 1}\right) & \leq H
        		\left(x_{1}^{k + 1} , x_{2}^{k}\right) + f_{2}\left(x_{2}^{k}\right) + \left(L_{2}(x_{1}^{k 
        		+ 1}) + s_{2}^{k} - \tau_{2}^{k}\left(1 - \alpha_{2}^{k}\right)\right)\Delta_{2}^{k + 1} 
        		\nonumber \\
        		& + \frac{1}{s_{2}^{k}}\left(L_{2}(x_{1}^{k + 1})^{2}\left(\beta_{2}^{k}\right)^{2} + s_{2}
        		^{k}\tau_{2}^{k}\alpha_{2}^{k}\right)\Delta_{2}^{k}. \label{TechC1:2}
        \end{align}
		By adding \eqref{TechC1:1} and \eqref{TechC1:2} we get
		\begin{align*}
        		F\left(\bx^{k + 1}\right) & \leq F\left(\bx^{k}\right) + \frac{1}{s_{1}^{k}}\left(L_{1}
        		(x_{2}^{k})^{2}\left(\beta_{1}^{k}\right)^{2} + s_{1}^{k}\tau_{1}^{k}\alpha_{1}^{k}\right)
        		\Delta_{1}^{k} + \frac{1}{s_{2}^{k}}\left(L_{2}(x_{1}^{k + 1})^{2}\left(\beta_{2}^{k}
        		\right)^{2} + s_{2}^{k}\tau_{2}^{k}\alpha_{2}^{k}\right)\Delta_{2}^{k} \\
        		& + \left(L_{1}(x_{2}^{k}) + s_{1}^{k} - \tau_{1}^{k}\left(1 - \alpha_{1}^{k}\right)\right)
        		\Delta_{1}^{k + 1} + \left(L_{2}(x_{1}^{k + 1}) + s_{2}^{k} - \tau_{2}^{k}\left(1 - 
        		\alpha_{2}^{k}\right)\right)\Delta_{2}^{k + 1}.
        \end{align*}	
        This proves the desired result.
	\end{proof}
	Before we proceed and for the sake of simplicity of our developments we would like to chose the 
	parameters $s_{1}^{k}$ and $s_{2}^{k}$ for all $k \in \nn$. The best choice can be derived by 
	minimizing the right-hand side of \eqref{L:ProxIne:2} with respect to $s$. Simple computations 
	yields that the minimizer should be
	\begin{equation*}
		s = L_{h}\frac{\norm{u - w}}{\norm{u^{+} - u}},
	\end{equation*}
	where $u , u^{+} , w$ and $L_{h}$ are all in terms of Lemma \ref{L:ProxIne}. In Proposition 
	\ref{TechC1} we have used Lemma \ref{L:ProxIne} with the following choices $u = x_{1}^{k}$, $u^{+} = 
	x_{1}^{k + 1}$ and $w = z_{1}^{k}$. Thus
	\begin{equation*}
		s_{1}^{k} = L_{1}(x_{2}^{k})\frac{\norm{x_{1}^{k} - z_{1}^{k}}}{\norm{x_{1}^{k + 1} - x_{1}
		^{k}}} = L_{1}(x_{2}^{k})\beta_{1}^{k}\frac{\norm{x_{1}^{k} - x_{1}^{k - 1}}}{\norm{x_{1}^{k + 
		1} - x_{1}^{k}}},
	\end{equation*}
	where the last equality follows from step \eqref{iPALM:StepX1:2}. Thus, from now on, we will use the 
	following parameters:
	\begin{equation} \label{ParameterS}
		s_{1}^{k} = L_{1}(x_{2}^{k})\beta_{1}^{k} \quad \text{and} \quad s_{2}^{k} = L_{2}(x_{1}^{k + 
		1})\beta_{2}^{k}, \quad  \forall \,\, k \in \nn.
	\end{equation}
	An immediate consequence of this choice of parameters which combined with Proposition \ref{TechC1} 
	is recorded now.
	\begin{corollary} \label{C:TechC1}
		Suppose that Assumption \ref{AssumptionsA} holds. Let $\left\{ \left(x_{1}^{k} , x_{2}^{k}
		\right) \right\}_{k \in \nn}$ be a sequence generated by iPALM, then for all $k \in \nn$, we 
		have that
		\begin{align*}
        		F\left(\bx^{k + 1}\right) & \leq F\left(\bx^{k}\right) + \left(L_{1}(x_{2}^{k})\beta_{1}^{k}  
        		+ \tau_{1}^{k}\alpha_{1}^{k}\right)\Delta_{1}^{k} + \left(L_{2}(x_{1}^{k + 1})\beta_{2}^{k} 
        		+ \tau_{2}^{k}\alpha_{2}^{k}\right)\Delta_{2}^{k} \\
        		& + \left(\left(1 + \beta_{1}^{k}\right)L_{1}(x_{2}^{k}) - \tau_{1}^{k}\left(1 - \alpha_{1}
        		^{k}\right)\right)\Delta_{1}^{k + 1} + \left(\left(1 + \beta_{2}^{k}\right)L_{2}(x_{1}^{k + 
        		1}) - \tau_{2}^{k}\left(1 - \alpha_{2}^{k}\right)\right)\Delta_{2}^{k + 1}.
        \end{align*}	
	\end{corollary}	
	Similarly to iPiano, the iPALM algorithm generates a sequence which does not ensure that the 
	function values decrease between two successive elements of the sequence. Thus we can not obtain 
	condition (C1) of Theorem \ref{T:GeneralConvergence}. Following \cite{OCBP2014} we construct an 
	auxiliary function which do enjoy the property of function values decreases. Let $\Psi : \rr^{n_{1} 
	\times n_{2}} \times \rr^{n_{1} \times n_{2}} \rightarrow \left(-\infty , \infty\right]$ be the 
	auxiliary function which is defined as follows
	\begin{equation} \label{D:Psi}
		\Psi_{\delta_{1} , \delta_{2}}\left(\bu\right) := F\left(u_{1}\right) + \frac{\delta_{1}}{2}
		\norm{u_{11} - u_{21}}^{2} + \frac{\delta_{2}}{2}\norm{u_{12} - u_{22}}^{2},
	\end{equation}
	where $\delta_{1} , \delta_{2} > 0$, $u_{1} = \left(u_{11} , u_{12}\right) \in \rr^{n_{1}} \times 
	\rr^{n_{2}}$, $u_{2} = \left(u_{21} , u_{22}\right) \in \rr^{n_{1}} \times \rr^{n_{2}}$ and $\bu = 
	\left(u_{1} , u_{2}\right)$. 
\medskip

	Let $\left\{ \left(x_{1}^{k} , x_{2}^{k}\right) \right\}_{k \in \nn}$ be a sequence generated by 
	iPALM and denote, for all $k \in \nn$, $u_{1}^{k} = \left(x_{1}^{k} , x_{2}^{k}\right)$, $u_{2}^{k} 
	= \left(x_{1}^{k - 1} , x_{2}^{k - 1}\right)$ and $\bu^{k} = \left(u_{1}^{k} , u_{2}^{k}\right)$. We 
	will prove now that the sequence $\Seq{\bu}{k}$ and the function $\Psi$ defined above do satisfy 
	conditions (C1), (C2) and (C3) of Theorem \ref{T:GeneralConvergence}. We begin with proving 
	condition (C1). To this end we will show that there are choices of $\delta_{1} > 0$ and $\delta_{2} 
	> 0$, such that there exists $\rho_{1} > 0$ which satisfies
	\begin{equation*}
		\rho_{1}\norm{\bu^{k + 1} - \bu^{k}}^{2} \leq \Psi\left(\bu^{k}\right) - \Psi\left(\bu^{k + 1}
		\right).
	\end{equation*}
	It is easy to check that using the notations defined in \eqref{D:Delta}, we have, for all $k \in \nn
	$, that
	\begin{equation*}
        	\Psi\left(\bu^{k}\right) = F\left(\bx^{k}\right) + \frac{\delta_{1}}{2}\norm{x_{1}^{k} - x_{1}
        	^{k - 1}}^{2} + \frac{\delta_{2}}{2}\norm{x_{2}^{k} - x_{2}^{k - 1}}^{2} = F\left(\bx^{k}\right) 
        	+ \delta_{1}	\Delta_{1}^{k} + \delta_{2}\Delta_{2}^{k}.
	\end{equation*}
	In order to prove that the sequence $\left\{ \Psi\left(\bu^{k}\right) \right\}_{k \in \nn}$ 
	decreases we will need the following technical result (we provide the proof in Appendix
	\ref{A:Technical}).
	\begin{lemma} \label{L:Technical}
		Consider the functions $g : \rr_{+}^{5} \rightarrow \rr$ and $h : \rr_{+}^{5} \rightarrow \rr$ 
		defined as follow
		\begin{align*}
			g\left(\alpha , \beta , \delta , \tau , L\right) & = \tau\left(1 - \alpha\right) - \left(1 + 
			\beta\right)L - \delta, \\
			h\left(\alpha , \beta , \delta , \tau , L\right) & = \delta - \tau\alpha - L\beta.
		\end{align*}
		Let $\varepsilon > 0$ and $\bar{\alpha} > 0$ be two real numbers for which $0 \leq \alpha \leq 
		\bar{\alpha} < 0.5\left(1 - \varepsilon\right)$. Assume, in addition, that $0 \leq L \leq 
		\lambda$ for some $\lambda > 0$ and $0 \leq \beta \leq \bar{\beta}$ with $\bar{\beta} > 0$. If
		\begin{align}
			\delta_{\ast} & = \frac{\bar{\alpha} + \bar{\beta}}{1 - \varepsilon - 2\bar{\alpha}}\lambda, 
			\label{L:Technical:1} \\
			\tau_{\ast} & = \frac{\left(1 + \varepsilon\right)\delta_{\ast} + \left(1 + \beta\right)L}{1 
			- \alpha}, \label{L:Technical:2}
		\end{align}
		then $g\left(\alpha , \beta , \delta_{\ast} , \tau_{\ast} , L\right) = \varepsilon\delta_{\ast}$ 
		and $h\left(\alpha , \beta , \delta_{\ast} , \tau_{\ast} , L\right) \geq \varepsilon
		\delta_{\ast}$.		
	\end{lemma}
	Based on the mentioned lemma we will set, from now on, the parameters $\tau_{1}^{k}$ and $\tau_{2}
	^{k}$ for all $k \in \nn$, as follow
\vspace{0.1in}

	{\center\fbox{\parbox{16cm}{
		\begin{equation} \label{ParameterTau}
			\tau_{1}^{k} = \frac{\left(1 + \varepsilon\right)\delta_{1}^{k} + \left(1 + \beta_{1}^{k}
			\right)L_{1}(x_{2}^{k})}{1 - \alpha_{1}^{k}} \quad \text{and} \quad \tau_{2}^{k} = 
			\frac{\left(1 + \varepsilon\right)\delta_{2}^{k} + \left(1 + \beta_{2}^{k}\right)L_{2}(x_{1}
			^{k + 1})}{1 - \alpha_{2}^{k}}. 
		\end{equation}}}}
\vspace{0.2in}	

	\begin{remark} \label{R:StepConvex}
		If we additionally know that $f_{i}$, $i = 1 , 2$, is convex, then a tighter bound can be used 
		in Proposition \ref{L:ProxIne} as described in Remark \ref{R:ProxIneConvex}. Using the tight 
		bound will improve the possible parameter $\tau_{i}^{k}$ that can be used (\cf 
		\eqref{ParameterTau}). Indeed, in the convex case, \eqref{L:Technical:1} and 
		\eqref{L:Technical:2} are given by
		\begin{align}
			\delta_{\ast} & = \frac{\bar{\alpha} + 2\bar{\beta}}{2\left(1 - \varepsilon - \bar{\alpha}
			\right)}\lambda, \label{R:StepConvex:1} \\
			\tau_{\ast} & = \frac{\left(1 + \varepsilon\right)\delta_{\ast} + \left(1 + \beta\right)L}
			{2 - \alpha}, \label{R:StepConvex:2}
		\end{align}		
		Thus, the parameters $\tau_{i}^{k}$, $i = 1 , 2$, can be taken in the convex case as follows
		\begin{equation*}
			\tau_{1}^{k} = \frac{\left(1 + \varepsilon\right)\delta_{1}^{k} + \left(1 + \beta_{1}^{k}
			\right)L_{1}(x_{2}^{k})}{2 - \alpha_{1}^{k}} \quad \text{and} \quad \tau_{2}^{k} = 
			\frac{\left(1 + \varepsilon\right)\delta_{2}^{k} + \left(1 + \beta_{2}^{k}\right)L_{2}(x_{1}
			^{k + 1})}{2 - \alpha_{2}^{k}}. 
		\end{equation*}
		This means that in the convex case, we can take smaller $\tau_{i}^{k}$, $i = 1 , 2$, which means  
		larger step-size in the algorithm. On top of that, in the case that $f_{i}$, $i = 1 , 2$, is 
		convex, it should be noted that a careful analysis shows that in this case the parameters 
		$\alpha_{i}^{k}$, $i = 1 , 2$, can be in the interval $\left[0 , 1\right)$ and not $\left[0 , 
		0.5\right)$ as stated in Lemma \ref{L:Technical} (see also Assumption \ref{AssumptionB} below).
	\end{remark}
	In order to prove condition C1 and according to Lemma \ref{L:Technical}, we will need to restrict 
	the possible values of the parameters $\alpha_{i}^{k}$ and $\beta_{i}^{k}$, $i = 1 , 2$, for all $k 
	\in \nn$. The following assumption is essential for our analysis.
	\begin{assumption} \label{AssumptionB}
		Let $\varepsilon > 0$ be an arbitrary small number. For all $k \in \nn$ and $i = 1 , 2$, there 
		exist $0 < {\bar \alpha_{i}} < \left(1/2\right)\left(1 - \varepsilon\right)$ such that $0 \leq 
		\alpha_{i}^{k} \leq {\bar \alpha_{i}}$. In addition, $0 \leq \beta_{i}^{k} \leq {\bar \beta_{i}}
		$ for some $\bar{\beta_{i}} > 0$.
	\end{assumption}	
	\begin{remark} \label{R:BoundedTau}
		It should be noted that using Assumption \ref{AssumptionB}, we obtain that $\tau_{1}^{k} \leq 
		\tau_{1}^{+}$ where
		\begin{equation*}
			\tau_{1}^{+} = \frac{\left(1 + \varepsilon\right)\delta_{1} + \left(1 + \bar{\beta_{1}}
			\right)\lambda_{1}^{+}}{1 - {\bar \alpha_{1}}},
		\end{equation*}	
		where $\delta_{1}$ is given in \eqref{L:Technical:1}. Similar arguments show that
		\begin{equation*}
			\tau_{2}^{k} \leq \tau_{2}^{+} := \frac{\left(1 + \varepsilon\right)\delta_{2} + 
			\left(1 + \bar{\beta_{2}}\right)\lambda_{2}^{+}}{1 - {\bar \alpha_{2}}}.
		\end{equation*}	
	\end{remark}
	Now we will prove a descent property of $\left\{ \Psi\left(\bu^{k}\right) \right\}_{k \in \nn}$.
	\begin{proposition} \label{P:C1}
		Let $\Seq{\bx}{k}$ be a sequence generated by iPALM which is assumed to be bounded. Suppose that 
		Assumptions \ref{AssumptionsA} and \ref{AssumptionB} hold true. Then, for all $k \in \nn$ and 
		$\varepsilon > 0$, we have
        \begin{equation*}
        		\rho_{1}\norm{\bu^{k + 1} - \bu^{k}}^{2} \leq \Psi\left(\bu^{k}\right) - \Psi\left(\bu^{k + 
        		1}\right),
        \end{equation*}
        where $\bu^{k} = \left(\bx^{k} , \bx^{k - 1}\right)$, $k \in \nn$ and $\rho_{1} = 
        \left(\varepsilon/2\right)\min \left\{ \delta_{1} , \delta_{2} \right\}$ with
		\begin{equation} \label{P:C1:1}
			\delta_{1} = \frac{{\bar \alpha_{1}} + \bar{\beta_{1}}}{1 - \varepsilon - 2{\bar 
			\alpha_{1}}}\lambda_{1}^{+} \quad \text{and} \quad \delta_{2} = \frac{{\bar \alpha_{2}} + 
			\bar{\beta_{2}}}{1 - \varepsilon - 2{\bar \alpha_{2}}}\lambda_{2}^{+}.
		\end{equation}		    
	\end{proposition}
	\begin{proof}
		From the definition of $\Psi$ (see \eqref{D:Psi}) and Corollary \ref{C:TechC1} we obtain that
		\begin{align*}
        		\Psi\left(\bu^{k}\right) - \Psi\left(\bu^{k + 1}\right) & = F\left(\bx^{k}\right) +
        		\delta_{1}\Delta_{1}^{k} + \delta_{2}\Delta_{2}^{k} - F\left(\bx^{k + 1}\right) - \delta_{1}
        		\Delta_{1}^{k + 1} - \delta_{2}\Delta_{2}^{k + 1} \\
        		& \geq \left(\tau_{1}^{k}\left(1 - \alpha_{1}^{k}\right) - \left(1 + \beta_{1}^{k}
        		\right)L_{1}(x_{2}^{k}) - \delta_{1}\right)\Delta_{1}^{k + 1} \nonumber \\
 	        	& + \left(\delta_{1} - \tau_{1}^{k}\alpha_{1}^{k} - L_{1}(x_{2}^{k})\beta_{1}^{k}\right)
 	        	\Delta_{1}^{k} \nonumber \\
        		& + \left(\tau_{2}^{k}\left(1 - \alpha_{2}^{k}\right) - \left(1 + \beta_{2}^{k}\right)L_{2}
        		(x_{1}^{k + 1}) - \delta_{2}\right)\Delta_{2}^{k + 1} \nonumber \\
        		& + \left(\delta_{2} - \tau_{2}^{k}\alpha_{2}^{k} - L_{2}(x_{1}^{k + 1})\beta_{2}^{k}\right)
        		\Delta_{2}^{k} \\
			& = a_{1}^{k}\Delta_{1}^{k + 1} + b_{1}^{k}\Delta_{1}^{k} + a_{2}^{k}\Delta_{2}^{k + 1} + 
			b_{2}^{k}\Delta_{2}^{k},        
		\end{align*}	
		where 
		\begin{align*}
        		a_{1}^{k} & := \tau_{1}^{k}\left(1 - \alpha_{1}^{k}\right) - \left(1 + \beta_{1}^{k}
        		\right)L_{1}(x_{2}^{k}) - \delta_{1} \quad \text{and} \quad b_{1}^{k} := \delta_{1} - 
        		\tau_{1}^{k}\alpha_{1}^{k} - L_{1}(x_{2}^{k})\beta_{1}^{k}, \nonumber \\
        		a_{2}^{k} & := \tau_{2}^{k}\left(1 - \alpha_{2}^{k}\right) - \left(1 + \beta_{2}^{k}
        		\right)L_{2}(x_{1}^{k + 1}) - \delta_{2} \quad \text{and} \quad b_{2}^{k} := \delta_{2} - 
        		\tau_{2}^{k}\alpha_{2}^{k} - L_{2}(x_{1}^{k + 1})\beta_{2}^{k}.	
        	\end{align*}	        
        Let $\varepsilon > 0$ be an arbitrary. Using \eqref{ParameterTau} and \eqref{P:C1:1} with the 
        notations of Lemma \ref{L:Technical} we immediately see that $a_{1}^{k} = g_{1}\left(\alpha_{1}
        ^{k} , \beta_{1}^{k} , \delta_{1} , \tau_{1}^{k} , L_{1}(x_{2}^{k})\right)$ and $a_{2}^{k} = 
        g_{2}\left(\alpha_{2}^{k} , \beta_{2}^{k} , \delta_{2} , \tau_{2}^{k} , L_{2}(x_{1}^{k + 1})
        \right)$. From Assumptions \ref{AssumptionsA} and \ref{AssumptionB} we get that the requirements 
        of Lemma \ref{L:Technical} are fulfilled, which means that Lemma \ref{L:Technical} can be 
        applied. Thus $a_{1}^{k} = \varepsilon\delta_{1}$ and $a_{2}^{k} = \varepsilon\delta_{2}$. Using 
        again the notions of Lemma \ref{L:Technical}, we have that $b_{1}^{k} = h_{1}\left(\alpha_{1}
        ^{k} , \beta_{1}^{k} , \delta_{1} , \tau_{1}^{k} , L_{1}(x_{2}^{k})\right)$ and $b_{2}^{k} = 
        h_{2}\left(\alpha_{2}^{k} , \beta_{2}^{k} , \delta_{2} , \tau_{2}^{k} , L_{2}(x_{1}^{k + 1})
        \right)$. Thus we obtain from Lemma \ref{L:Technical} that $b_{1}^{k} \geq \varepsilon\delta_{1}
        $ and $b_{2}^{k} \geq \varepsilon\delta_{2}$. Hence, for $\rho_{1} = \left(\varepsilon/2\right) 
        \min \left\{ \delta_{1} , \delta_{2} \right\}$, we have
		\begin{align*}
        		\Psi\left(\bu^{k}\right) - \Psi\left(\bu^{k + 1}\right) & \geq a_{1}^{k}\Delta_{1}^{k + 1} + 
        		b_{1}^{k}\Delta_{1}^{k} + a_{2}^{k}\Delta_{2}^{k + 1} + b_{2}^{k}\Delta_{2}^{k} \\
			& \geq \varepsilon\delta_{1}\left(\Delta_{1}^{k + 1} + \Delta_{1}^{k}\right) + \varepsilon
			\delta_{1}\left(\Delta_{2}^{k + 1} + \Delta_{2}^{k}\right) \\		
			& \geq \rho_{1}\left(\Delta_{1}^{k + 1} + \Delta_{1}^{k}\right) + \rho_{1}\left(\Delta_{2}
			^{k + 1} + \Delta_{2}^{k}\right) \\		
			& = \rho_{1}\norm{\bu^{k + 1} - \bu^{k}}^{2},	
		\end{align*}	
		where the last equality follows from \eqref{D:Delta}. This completes the proof.
	\end{proof}
	Now, we will prove that condition (C2) of Theorem \ref{T:GeneralConvergence} holds true for the 
	sequence $\Seq{\bu}{k}$ and the function $\Psi$.
	\begin{proposition} \label{P:ConC2}
		Let $\Seq{\bx}{k}$ be a sequence generated by iPALM which is assumed to be bounded. Suppose that 
		Assumptions \ref{AssumptionsA} and \ref{AssumptionB} hold true. Assume that $\bu^{k} = 
		\left(\bx^{k} , \bx^{k - 1}\right)$, $k \in \nn$. Then, there exists a positive scalar $\rho_{2}
		$ such that for some $\bw^{k} \in \partial \Psi\left(\bu^{k}\right)$ we have
        \begin{equation*}
        		\norm{\bw^{k}} \leq \rho_{2}\norm{\bu^{k} - \bu^{k - 1}}.
       	\end{equation*}
	\end{proposition}
	\begin{proof}
		Let $k \geq 2$. By the definition of $\Psi$ (see \eqref{D:Psi}) we have that
        \begin{align*}
        		\partial \Psi\left(\bu^{k}\right) = \left(\partial_{x_{1}} F\left(\bx^{k}\right) + 
        		\delta_{1}\left(x_{1}^{k} - x_{1}^{k - 1}\right) , \partial_{x_{2}} F\left(\bx^{k}\right) + 
        		\delta_{2}\left(x_{2}^{k} - x_{2}^{k - 1}\right) , \delta_{1}\left(x_{1}^{k - 1} - x_{1}
        		^{k}\right), \right. \\
        		& \hspace{-0.99in} \left. \delta_{2}\left(x_{2}^{k - 1} - x_{2}^{k}\right)\right).
       	\end{align*}
       	By the definition of $F$ (see \eqref{Model}) and \cite[Proposition 1, Page 465]{BST2014} we get 
       	that 
       	\begin{equation} \label{P:ConC2:1}
        		\partial F\left(\bx^{k}\right) = \left(\partial f_{1}\left(x_{1}^{k}\right) + \nabla_{x_{1}} 
        		H\left(x_{1}^{k} , x_{2}^{k}\right) , \partial f_{2}\left(x_{2}^{k}\right) + \nabla_{x_{2}} 
        		H\left(x_{1}^{k} , x_{2}^{k}\right)\right).
       	\end{equation}
       	From the definition of the proximal mapping (see \eqref{D:ProximalMap}) and the iterative step 
       	\eqref{iPALM:StepX1:3} we have
        \begin{equation*}
            x_{1}^{k} \in \argmin_{x_{1} \in \rr^{n_{1}}} \left\{ \act{x_{1} - y_{1}^{k - 1} , 
            \nabla_{x_{1}} H\left(z_{1}^{k - 1} , x_{2}^{k - 1}\right)} + \frac{\tau_{1}^{k - 1}}{2}
            \norm{x_{1} - y_{1}^{k - 1}}^{2} + f_{1}\left(x_{1}\right) \right\}.
        \end{equation*}
        Writing down the optimality condition yields
        \begin{equation*}
            \nabla_{x_{1}} H\left(z_{1}^{k - 1} , x_{2}^{k - 1}\right) + \tau_{1}^{k - 1}\left(x_{1}^{k} 
            - y_{1}^{k - 1}\right) + \xi_{1}^{k}  = 0,
        \end{equation*}
        where $\xi_{1}^{k} \in \partial f_{1}\left(x_{1}^{k}\right)$. Hence
        \begin{equation*}
            \nabla_{x_{1}} H\left(z_{1}^{k - 1} , x_{2}^{k - 1}\right) + \xi_{1}^{k} = \tau_{1}^{k - 1}
            \left(y_{1}^{k - 1} - x_{1}^{k}\right) = \tau_{1}^{k - 1}\left(x_{1}^{k - 1} - x_{1}^{k} + 
            \alpha_{1}^{k - 1}\left(x_{1}^{k - 1} - x_{1}^{k - 2}\right)\right),
        \end{equation*}
        where the last equality follows from step \eqref{iPALM:StepX1:1}. By defining
        \begin{equation} \label{Item4-1}
            v_{1}^{k} := \nabla_{x_{1}} H\left(x_{1}^{k} , x_{2}^{k}\right) - \nabla_{x_{1}} 
            H\left(z_{1}^{k - 1} , x_{2}^{k - 1}\right) + \tau_{1}^{k - 1}\left(x_{1}^{k - 1} - x_{1}
            ^{k} + \alpha_{1}^{k - 1}\left(x_{1}^{k - 1} - x_{1}^{k - 2}\right)\right), 
        \end{equation}
        we obtain from \eqref{P:ConC2:1} that $v_{1}^{k} \in \partial_{x_{1}} F\left(\bx^{k}\right)$. 
        Similarly, from the iterative step \eqref{iPALM:StepX2:3}, by defining
        \begin{equation} \label{Item4-2}
            v_{2}^{k} := \nabla_{x_{2}} H\left(x_{1}^{k} , x_{2}^{k}\right) - \nabla_{x_{2}} 
            H\left(x_{1}^{k} , z_{2}^{k - 1}\right) + \tau_{2}^{k - 1}\left(x_{2}^{k - 1} - x_{2}^{k} + 
            \alpha_{2}^{k - 1}\left(x_{2}^{k - 1} - x_{2}^{k - 2}\right)\right), 
        \end{equation}
        we have that $v_{2}^{k} \in \partial_{x_{2}} F\left(\bx^{k}\right)$.
\medskip

		Thus, for 
		\begin{equation*}
			\bw^{k} := \left(v_{1}^{k} + \delta_{1}\left(x_{1}^{k} - x_{1}^{k - 1}\right) , v_{2}^{k} + 
			\delta_{2}\left(x_{2}^{k} - x_{2}^{k - 1}\right)\right) , \delta_{1}\left(x_{1}^{k} - x_{1}
			^{k - 1} , \delta_{2}\left(x_{2}^{k} - x_{2}^{k - 1}\right)\right),
		\end{equation*}
		we obtain that
       	\begin{equation} \label{P:ConC2:0}
        		\norm{\bw^{k}} \leq \norm{v_{1}^{k}} + \norm{v_{2}^{k}} + 2\delta_{1}\norm{x_{1}^{k} - x_{1}
        		^{k - 1}} + 2\delta_{2}\norm{x_{2}^{k} - x_{2}^{k - 1}}.
       	\end{equation}
       	This means that we have to bound from above the norms of $v_{1}^{k}$ and $v_{2}^{k}$. Since 
       	$\nabla H$ is Lipschitz continuous on bounded subsets of $\rr^{n_{1}} \times \rr^{n_{2}}$ (see 
       	Assumption \ref{AssumptionsA}(v)) and since we assumed that $\Seq{\bx}{k}$ is bounded, there 
       	exists $M > 0$ such that
        \begin{align*}
            \norm{v_{1}^{k}} & \leq \tau_{1}^{k - 1}\norm{x_{1}^{k - 1} - x_{1}^{k} + \alpha_{1}^{k - 1}
            \left(x_{1}^{k - 1} - x_{1}^{k - 2}\right)} + \norm{\nabla_{x_{1}} H\left(x_{1}^{k} , x_{2}
            ^{k}\right) - \nabla_{x_{1}} H\left(z_{1}^{k - 1} , x_{2}^{k - 1}\right)} \\
            & \leq \tau_{1}^{k - 1}\norm{x_{1}^{k - 1} - x_{1}^{k}} + \tau_{1}^{k - 1}\alpha_{1}^{k - 1}
            \norm{x_{1}^{k - 1} - x_{1}^{k - 2}} + M\norm{\bx^{k} - \left(z_{1}^{k - 1} , x_{2}^{k - 1}
            \right)} \\
            & \leq \tau_{1}^{+}\left(\norm{x_{1}^{k - 1} - x_{1}^{k}} + \norm{x_{1}^{k - 1} - x_{1}^{k - 
            2}}\right) + M\norm{\left(x_{1}^{k} - x_{1}^{k - 1} - \beta_{1}^{k - 1}\left(x_{1}^{k - 1} - 
            x_{1}^{k - 2}\right) , x_{2}^{k} - x_{2}^{k - 1}\right)}, \\
            & = \tau_{1}^{+}\left(\norm{x_{1}^{k - 1} - x_{1}^{k}} + \norm{x_{1}^{k - 1} - x_{1}^{k - 
            2}}\right) + M\norm{\left(\bx^{k} - \bx^{k - 1}\right) - \beta_{1}^{k - 1}\left(x_{1}^{k - 
            1} - x_{1}^{k - 2} , \bo\right)} \\
           	& \leq \left(\tau_{1}^{+} + M\right)\left(\norm{\bx^{k} - \bx^{k - 1}} + \norm{\bx^{k - 1} - 
           	\bx^{k - 2}}\right),
        \end{align*}
        where the third inequality follows from \eqref{iPALM:StepX1:2}, the fact the sequence $\left\{ 
        \tau_{1}^{k} \right\}_{k \in \nn}$ is bounded from above by $\tau_{1}^{+}$ (see Remark 
        \ref{R:BoundedTau}) and $\alpha_{1}^{k} , \beta_{1}^{k} \leq 1$ for all $k \in \nn$. On the 
        other hand, from the Lipschitz continuity of $\nabla_{x_{2}} H\left(x_{1} , \cdot\right)$ (see 
        Assumption \ref{AssumptionsA}(iii)), we have that
        \begin{align*}
            \norm{v_{2}^{k}} & \leq \tau_{2}^{k - 1}\norm{x_{2}^{k - 1} - x_{2}^{k} + \alpha_{2}^{k - 1}
            \left(x_{2}^{k - 1} - x_{2}^{k - 2}\right)} + \norm{\nabla_{x_{2}} H\left(x_{1}^{k} , x_{2}
            ^{k}\right) - \nabla_{x_{2}} H\left(x_{1}^{k} , z_{2}^{k - 1}\right)} \\
            & \leq \tau_{2}^{k - 1}\norm{x_{2}^{k - 1} - x_{2}^{k}} + \tau_{2}^{k - 1}\alpha_{2}^{k - 1}
            \norm{x_{2}^{k - 1} - x_{2}^{k - 2}} + L_{1}(x_{1}^{k})\norm{x_{2}^{k} - z_{2}^{k - 1}} \\
            & \leq \tau_{2}^{+}\left(\norm{x_{2}^{k - 1} - x_{2}^{k}} + \norm{x_{2}^{k - 1} - x_{2}^{k - 
            2}}\right) + \lambda_{2}^{+}\norm{x_{2}^{k} - x_{2}^{k - 1} - \beta_{2}^{k - 1}\left(x_{2}
            ^{k - 1} - x_{2}^{k - 2}\right)} \\
            & \leq \left(\tau_{2}^{+} + \lambda_{2}^{+}\right)\left(\norm{x_{2}^{k - 1} - x_{2}^{k}} + 
            \norm{x_{2}^{k - 1} - x_{2}^{k - 2}}\right) \\
           	& \leq \left(\tau_{2}^{+} + \lambda_{2}^{+}\right)\left(\norm{\bx^{k} - \bx^{k - 1}} + 
           	\norm{\bx^{k - 1} - \bx^{k - 2}}\right),
        \end{align*}
        where the third and fourth inequalities follow from \eqref{iPALM:StepX2:2}, the fact the 
        sequence $\left\{ \tau_{2}^{k} \right\}_{k \in \nn}$ is bounded from above by $\tau_{2}^{+}$ 
        (see Remark \ref{R:BoundedTau}) and $\beta_{2}^{k} \leq 1$ for all $k \in \nn$. Summing up these 
        estimations, we get from \eqref{P:ConC2:0} that
		\begin{align*}
        		\norm{\bw^{k}} & \leq \norm{v_{1}^{k}} + \norm{v_{2}^{k}} + 2\delta_{1}\norm{x_{1}^{k} - 
        		x_{1}^{k - 1}} + 2\delta_{2}\norm{x_{2}^{k} - x_{2}^{k - 1}} \\
        		& \leq \norm{v_{1}^{k}} + \norm{v_{2}^{k}} + 2\left(\delta_{1} + \delta_{2}\right)
        		\norm{\bx^{k} - \bx^{k - 1}} \\
        		& \leq \left(\tau_{1}^{+} + M + \tau_{2}^{+} + \lambda_{2}^{+}\right)\left(\norm{\bx^{k} - 
        		\bx^{k - 1}} + \norm{\bx^{k - 1} - \bx^{k - 2}}\right) + 2\left(\delta_{1} + \delta_{2}
        		\right)\norm{\bu^{k} - \bu^{k - 1}} \\        		
        		& \leq \left(\sqrt{2}\left(\tau_{1}^{+} + M + \tau_{2}^{+} + \lambda_{2}^{+}\right) + 
        		2\left(\delta_{1} + \delta_{2}\right)\right)\norm{\bu^{k} - \bu^{k - 1}}, 		
       	\end{align*}        
       	where the second inequality follows from the fact that $\norm{x_{i}^{k} - x_{i}^{k - 1}} \leq 
       	\norm{\bx^{k} - \bx^{k - 1}}$ for $i = 1 , 2$, the third inequality follows from the fact that 
       	$\norm{\bx^{k} - \bx^{k - 1}} \leq \norm{\bu^{k} - \bu^{k - 1}}$ and the last inequality follows 
       	from the fact that $\norm{\bx^{k} - \bx^{k - 1}} + \norm{\bx^{k - 1} - \bx^{k - 2}} \leq 
       	\sqrt{2}\norm{\bu^{k} - \bu^{k - 1}}$. This completes the proof with $\rho_{2} = \sqrt{2}
       	\left(\tau_{1}^{+} + M + \tau_{2}^{+} + \lambda_{2}^{+}\right) + 2\left(\delta_{1} + \delta_{2}
       	\right)$.
	\end{proof}	
	So far we have proved that the sequence $\Seq{\bu}{k}$ and the function $\Psi$ (see \eqref{D:Psi}) 
	satisfy conditions (C1) and (C2) of Theorem \ref{T:GeneralConvergence}. Now, in order to get that 
	$\Seq{\bu}{k}$ converges to a critical point of $\Psi$, it remains to prove that condition (C3) 
	holds true.
	\begin{proposition} \label{P:ConC3}
		Let $\Seq{\bx}{k}$ be a sequence generated by iPALM which is assumed to be bounded. Suppose that 
		Assumptions \ref{AssumptionsA} and \ref{AssumptionB} hold true. Assume that $\bu^{k} = 
		\left(\bx^{k} , \bx^{k - 1}\right)$, $k \in \nn$. Then, each limit point in the set $\omega
		\left(\bu^{0}\right)$ is a critical point of $\Psi$.
	\end{proposition}
	\begin{proof}
		Since $\Seq{\bu}{k}$ is assumed to be bounded, the set $\omega\left(\bu^{0}\right)$ is nonempty. 
		Thus there exists $\bu^{\ast} = \left(x_{1}^{\ast} , x_{2}^{\ast} , {\hat x_{1}} , {\hat x_{2}}
		\right)$ which is a limit point of $\left\{ \bu^{k_{l}} \right\}_{l\in \nn}$, which is a 
		subsequence of $\Seq{\bu}{k}$. We will prove that $\bu^{\ast}$ is a critical point of $\Psi$ 
		(see \eqref{D:Psi}). From condition (C2), for some $\bw^{k} \in \partial \Psi\left(\bu^{k}
		\right)$, we have that 
       	\begin{equation*}
        		\norm{\bw^{k}} \leq \rho_{2}\norm{\bu^{k} - \bu^{k - 1}}.
        \end{equation*}
        From Proposition \ref{P:C1}, it follow that for any $N \in \nn$, we have
        \begin{equation} \label{P:ConC3:1}
        		\rho_{1}\sum_{k = 0}^{N} \norm{\bu^{k + 1} - \bu^{k}}^{2} \leq \Psi\left(\bu^{0}\right) - 
        		\Psi\left(\bu^{N + 1}\right).
        \end{equation}
        Since $F$ is bounded from below (see Assumption \ref{AssumptionsA}(ii)) and the fact that $\Psi
        \left(\cdot\right) \geq F\left(\cdot\right)$ we obtain that $\Psi$ is also bounded from below. 
        Thus, letting $N \rightarrow \infty$ in \eqref{P:ConC3:1} yields that 
        \begin{equation} \label{P:ConC3:2}
        		\sum_{k = 0}^{\infty} \norm{\bu^{k + 1} - \bu^{k}}^{2} < \infty,
        \end{equation}
        which means that
        \begin{equation} \label{P:ConC3:3}
        		\lim_{k \rightarrow \infty} \norm{\bu^{k + 1} - \bu^{k}} = 0.
        \end{equation}
		This fact together with condition (C1) implies that $\norm{\bw^{k}} \rightarrow 0$ as $k 
		\rightarrow \infty$. Thus, in order to use the closedness property of $\partial \Psi$ (see 
		Remark \ref{R:SubdiffClosed}) we remain to show that $\left\{ \Psi\left(\bu^{k}\right) \right\}
		_{k \in \nn}$ converges to $\Psi\left(\bu^{\ast}\right)$. Since $f_{1}$ and $f_{2}$ are lower 
		semicontinuous (see Assumption \ref{AssumptionsA}(i)), we obtain that
    		\begin{equation} \label{L:Critical:2}
    			\limitinf{k}{\infty} f_{1}\left(x_{1}^{k}\right) \geq f_{1}\left(x_{1}^{\ast}\right) \quad 
    			\text{and} \quad \limitinf{k}{\infty} f_{2}\left(x_{2}^{k}\right) \geq f_{2}\left(x_{2}
    			^{\ast}\right).
	    \end{equation}
    	 	From the iterative step \eqref{iPALM:StepX1:3}, we have, for all integer $k$, that
     	\begin{equation*}
     		x_{1}^{k + 1} \in \argmin_{x_{1} \in \rr^{n_{1}}} \left\{ \act{x_{1} - y_{1}^{k} ,
     		\nabla_{x_{1}} H\left(z_{1}^{k} , x_{2}^{k}\right)} + \frac{\tau_{1}^{k}}{2}\norm{x_{1} - 
     		x_{1}^{k}}^{2} + f_{1}\left(x_{1}\right) \right\}.
	    \end{equation*}
     	Thus letting $x_{1} = x_{1}^{\ast}$ in the above, we get
     	\begin{align*}
     		\act{x_{1}^{k + 1} - y_{1}^{k} , \nabla_{x_{1}} H\left(z_{1}^{k} , x_{2}^{k}\right)} + 
     		\frac{\tau_{1}^{k}}{2}\norm{x_{1}^{k + 1} - x_{1}^{k}}^{2} + f_{1}\left(x_{1}^{k + 1}\right) 
     		\\
         	& \hspace{-2.3in} \leq \act{x_{1}^{\ast} - y_{1}^{k} , \nabla_{x_{1}} H\left(z_{1}^{k} , 
         	x_{2}^{k}\right)} + \frac{\tau_{1}^{k}}{2}\norm{x_{1}^{\ast} - x_{1}^{k}}^{2} + f_{1}
         	\left(x_{1}^{\ast}\right).
     	\end{align*}
     	Choosing $k = k_{l} - 1$ and letting $k$ goes to infinity, we obtain
     	\begin{align}
     		\limsup_{l \rightarrow \infty} f_{1}\left(x_{1}^{k_{l}}\right) & \leq \limsup_{l 
     		\rightarrow \infty} \left(\act{x_{1}^{\ast} - x_{1}^{k_{l}} , \nabla_{x_{1}} H\left(z_{1}
     		^{k_{l} - 1} , x_{2}^{k_{l} - 1}\right)} + \frac{\tau_{1}^{k_{l} - 1}}{2}\norm{x_{1}^{\ast} 
     		- x_{1}^{k_{l} - 1}}^{2}\right) \nonumber \\
     		& + f_{1}\left(x_{1}^{\ast}\right), \label{L:Critical:3}
		\end{align}
        where we have used the facts that both sequences $\Seq{\bx}{k}$ (and therefore $\Seq{\bz}{k}$) 
        and $\left\{ \tau_{1}^{k} \right\}_{k \in \nn}$ (see Remark \ref{R:BoundedTau}) are bounded, 
        $\nabla H$ continuous and that the distance between two successive iterates tends to zero (see 
        \eqref{P:ConC3:3}). For that very reason we also have that $x_{1}^{k_{l}} \rightarrow x_{1}
        ^{\ast}$ as $l \rightarrow \infty$, hence \eqref{L:Critical:3} reduces to $\limsup_{l 
        \rightarrow \infty} f_{1}\left(x_{1}^{k_{l}}\right) \leq f_{1}\left(x_{1}^{\ast}\right)$. Thus, 
        in view of \eqref{L:Critical:2}, $f_{1}\left(x_{1}^{k_{l}}\right)$ tends to $f_{1}\left(x_{1}
        ^{\ast}\right)$ as $k \rightarrow \infty$. Arguing similarly with $f_{2}$ and $x_{2}^{k + 1}$ we 
        thus finally obtain from \eqref{P:ConC3:3} that
        \begin{align}
        		\lim_{l \rightarrow \infty} \Psi\left(\bu^{k_{l}}\right) & = \lim_{l \rightarrow \infty} 
        		\left\{ f_{1}\left(x_{1}^{k_{l}}\right) + f_{2}\left(x_{2}^{k_{l}}\right) + H
        		\left(\bx^{k_{l}}\right) + \frac{\delta_{1}}{2}\norm{x_{1}^{k_{l}} - x_{1}^{k_{l} - 1}}^{2} 
        		+ \frac{\delta_{2}}{2}\norm{x_{2}^{k_{l}} - x_{2}^{k_{l} - 1}}^{2} \right\} \nonumber \\
            & = f_{1}\left(x_{1}^{\ast}\right) + f_{2}\left(x_{2}^{\ast}\right) + H\left(x_{1}^{\ast} , 
            x_{2}^{\ast}\right) \nonumber \\
            & = F\left(x_{1}^{\ast} , x_{2}^{\ast}\right) \label{L:Critical:4} \\
            & = \Psi\left(\bu^{\ast}\right). \nonumber
        \end{align}
		Now, the closedness property of $\partial \Psi$ (see Remark \ref{R:SubdiffClosed}) implies that 
		$\bo \in \partial \Psi\left(\bu^{\ast}\right)$, which proves that $\bu^{\ast}$ is a critical 
		point of $\Psi$. This proves condition (C3).
	\end{proof}
	Now using the convergence proof methodology of \cite{BST2014} which is summarized in Theorem 
	\ref{T:GeneralConvergence} we can obtain the following result.
	\begin{corollary} \label{C:PsiConv}
		Let $\Seq{\bx}{k}$ be a sequence generated by iPALM which is assumed to be bounded. Suppose that 
		Assumptions \ref{AssumptionsA} and \ref{AssumptionB} hold true. Assume that $\bu^{k} = 
		\left(\bx^{k} , \bx^{k - 1}\right)$, $k \in \nn$. If $F$ is a KL function, then the sequence 
		$\Seq{\bu}{k}$ converges to a critical point $\bu^{\ast}$ of $\Psi$.
	\end{corollary}
	\begin{proof}
		The proof follows immediately from Theorem \ref{T:GeneralConvergence} since Proposition 
		\ref{P:C1} proves that condition (C1) holds true, Proposition \ref{P:ConC2} proves that 
		condition (C2) holds true, and condition (C3) was proved in Proposition \ref{P:ConC3}. It is 
		also clear that if $F$ is a KL function then obviously $\Psi$ is a KL function since we just add 
		two quadratic functions.
	\end{proof}
	To conclude the convergence theory of iPALM we have to show that the sequence $\Seq{\bx}{k}$ which 
	is generated by iPALM converges to a critical point of $F$ (see \eqref{Model}).
	\begin{theorem}[Convergence of iPALM]
		Let $\Seq{\bx}{k}$ be a sequence generated by iPALM which is assumed to be bounded. Suppose that 
		Assumptions \ref{AssumptionsA} and \ref{AssumptionB} hold true. If $F$ is a KL function, then 
		the sequence $\Seq{\bx}{k}$ converges to a critical point $\bx^{\ast}$ of $F$.
	\end{theorem}
	\begin{proof}
		From Corollary \ref{C:PsiConv} we have that the sequence $\Seq{\bu}{k}$ converges to a critical 
		point $\bu^{\ast} = \left(u_{11}^{\ast} , u_{12}^{\ast} , u_{21}^{\ast} , u_{22}^{\ast}\right)$ 
		of $\Psi$. Therefore, obviously also the sequence $\Seq{\bx}{k}$ converges. Let $\bx^{\ast}$ be 
		the limit point of $\Seq{\bx}{k}$. Thence $u_{1}^{\ast} = \left(u_{11}^{\ast} , u_{12}^{\ast}
		\right) = \bx^{\ast}$ and $u_{2}^{\ast} = \left(u_{21}^{\ast} , u_{22}^{\ast}\right) = 
		\bx^{\ast}$ (see the discussion on Page 14). We will prove that $\bx^{\ast}$ is a critical point 
		of $F$ (see \eqref{Model}), that is, we have to show that $0 \in \partial F\left(\bx^{\ast}
		\right)$. Since $\bu^{\ast}$ is a critical point of $\Psi$, it means that $0 \in \partial \Psi
		\left(\bu^{\ast}\right)$. Thus
		\begin{equation*}
			0 \in \left(\partial_{x_{1}} F\left(u_{1}^{\ast}\right) + \delta_{1}\left(u_{11}^{\ast} - 
			u_{12}^{\ast}\right) , \partial_{x_{2}} F\left(u_{2}^{\ast}\right) + \delta_{2}\left(u_{21}
			^{\ast} - u_{22}^{\ast}\right) , \delta_{1}\left(u_{11}^{\ast} - u_{12}^{\ast}\right), 
			\delta_{2}\left(u_{21}^{\ast} - u_{22}^{\ast}\right)\right),
		\end{equation*}			
		which means that
		\begin{equation*}
			0 \in \left(\partial_{x_{1}} F\left(u_{1}^{\ast}\right) , \partial_{x_{2}} F\left(u_{2}
			^{\ast}\right)\right) = \partial F\left(\bx^{\ast}\right).
		\end{equation*}			
		This proves that $\bx^{\ast}$ is a critical point of $F$.
	\end{proof}

\section{Numerical Results} \label{Sec:Numerics}
	In this section we consider several important applications in image processing and machine learning 
	to illustrate the numerical performance of the proposed iPALM method. All algorithms have been 
	implemented in Matlab R2013a and executed on a server with Xeon(R) E5-2680 v2 @ 2.80GHz CPUs and 
	running Linux.

\subsection{Non-Negative Matrix Factorization}
	In our first example we consider the problem of using the Non-negative Matrix Factorization (NMF) to 
	decompose a set of facial images into a number of sparse basis faces, such that each face of the 
	database can be approximated by using a small number of those parts. We use the ORL database 
	\cite{orl} that consists of $400$ normalized facial images. In order to enforce sparsity in the 
	basis faces, we additionally consider a $\ell_{0}$ sparsity constraint (see, for example, 
	\cite{Peharz}). The sparse NMF problem to be solved is given by
	\begin{equation}
  		\min_{B , C} \left\{ \frac{1}{2}\norm{A - BC}^{2} : \, B , C \geq 0, \,\, \norm{b_{i}}_{0} \leq 
  		s, \, i = 1, 2 , \ldots , r \right\},
	\end{equation}
	where $A \in \rr^{m \times n}$ is the data matrix, organized in a way that each column of the matrix 
	$A$ corresponds to one face of size $m = 64 \times 64$ pixels. In total, the matrix holds $n = 400$ 
	faces (see Figure \ref{fig:orl} for a visualization of the data matrix $A$). The matrix $B \in 
	\rr^{m \times r}$ holds the $r$ basis vectors $b_{i} \in \rr^{m \times 1}$, where $r$ corresponds to 
	the number of sparse basis faces. The sparsity constraint applied to each basis face requires that 
	the number of non-zero elements in each column vector $b_{i}$, $i = 1, 2 , \ldots , r$ should have 
	less or equal to $s$ non-zero elements. Finally, the matrix $C \in \rr^{r \times n}$ corresponds to 
	the coefficient vectors.

	\begin{figure}[t!]
  		\centering
  		\includegraphics[width=1.0\textwidth]{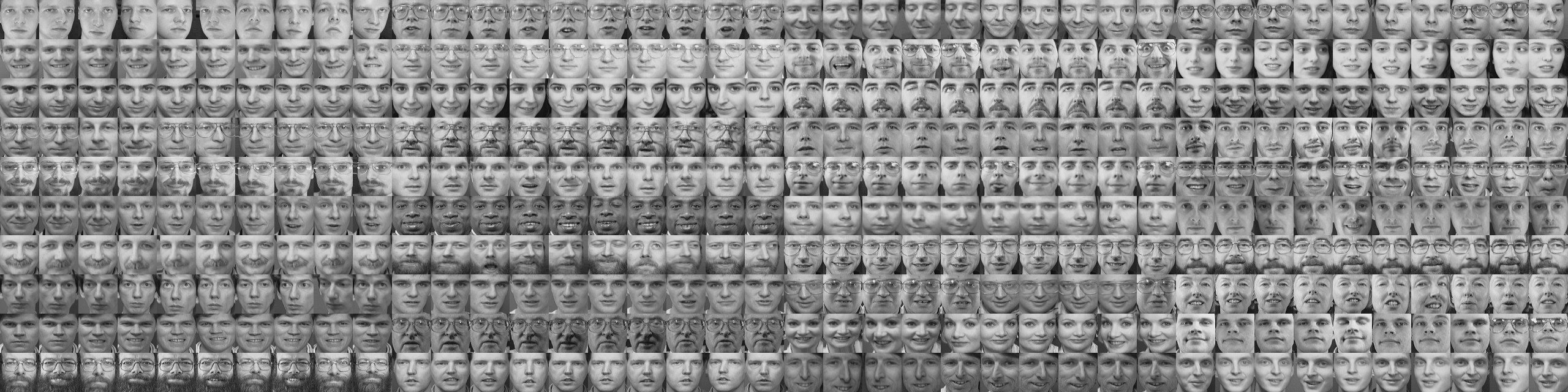}
  		\caption{ORL database which includs $400$ faces which we used in our NMF example.}		
  		\label{fig:orl}
	\end{figure}

	The application of the proposed iPALM algorithm to this problem is straight-forward. The first block 
	of variables corresponds to the matrix $B$ and the second block of variables corresponds to the 
	matrix $C$. Hence, the smooth coupling function of both blocks is given by
	\begin{equation*}
  		H\left(B , C\right) = \frac{1}{2}\norm{A - BC}^{2}.
	\end{equation*}
	The block-gradients and respective block-Lipschitz constants are easily computed via
	\begin{align*}
  		\nabla_{B} H\left(B , C\right) = \left(BC - A\right)C^{T}, & \quad L_{1}\left(C\right) = 
  		\norm{CC^{T}}_{2}, \\
  		\nabla_{C} H\left(B , C\right) = B^{T}\left(BC - A\right), & \quad L_{2}\left(B\right) = 
  		\norm{B^{T}B}_{2}.
	\end{align*}	
	The nonsmooth function for the first block, $f_{1}\left(B\right)$, is given by the non-negativity 
	constraint $B \geq 0$ and the $\ell_{0}$ sparsity constraint applied to each column of the matrix $B
	$, that is,
	\begin{equation*}
  		f_{1}\left(B\right) = 
  		\begin{cases} 
    			0, \quad B \geq 0, \, \norm{b_{i}}_{0} \leq s, \, i = 1, 2 , \ldots , r, \\
    			\infty, \,\,\, \text{else}.
  		\end{cases} 
	\end{equation*}
	Although this function is an indicator function of a nonconvex set, it is shown in \cite{BST2014}, 
	that its proximal mapping can be computed very efficiently (in fact in linear time) via
	\begin{equation*}
  		B = \prox_{f_{1}}\left(\hat{B}\right) \Leftrightarrow b_{i} = T_{s}\left(\hat{b_{i}^{+}}\right), 
  		\quad i = 1, 2 , \ldots , r,
	\end{equation*}	
	where $\hat{b_{i}^{+}} = \max\{ \hat{b_{i}} , 0 \}$ denotes an elementwise truncation at zero and 
	the operator $T_{s}\left(\hat{b_{i}^{+}}\right)$ corresponds to first sorting the values of $\hat 
	{b_{i}^{+}}$, keeping the $s$ largest values and setting the remaining $m-s$ values to zero.
\medskip

	The nonsmooth function corresponding to the second block is simply the indicator function of the 
	non-negativity constraint of $C$, that is,
	\begin{equation*}
  		f_{2}\left(C\right) = 
  		\begin{cases} 
    			0, \quad C \geq 0, \\
    			\infty, \,\,\, \text{else},
  		\end{cases}
	\end{equation*}
	and its proximal mapping is trivially given by
	\begin{equation*}
  		C = \prox_{f_{2}}\left(\hat{C}\right) = \hat{C^{+}},
	\end{equation*}
	which is again an elementwise truncation at zero.
\medskip

	In our numerical example we set $r = 25$, that is we seek for $25$ sparse basis images. Figure
	\ref{fig:basis-faces} shows the results of the basis faces when running the iPALM algorithm for 
	different sparsity settings. One can see that for smaller values of $s$, the algorithm leads to more 
	compact representations. This might improve the generalization capabilities of the representation.
\medskip

	In order to investigate the properties of iPALM on the inertial parameters, we run the algorithm for 
	a specific sparsity setting ($s = 33\%$) using different constant settings of $\alpha_{i}$ and 
	$\beta_{i}$ for $i = 1 , 2$. From our convergence theory (see Proposition \ref{P:C1}) it follows 
	that the parameters $\delta_{i}$, $i = 1 , 2$, have to be chosen as constants. However, in practice 
	we shall use a varying parameter $\delta_{i}^{k}$, $i = 1 , 2$, and assume that the parameters will 
	become constant after a certain number of iterations. Observe, that the nonsmooth function of the 
	first block is nonconvex ($\ell_{0}$ constraint), while the nonsmooth function of the second block 
	is convex (non-negativity constraint) which will affect the rules to compute the parameters (see 
	Remark \ref{R:StepConvex}).
\medskip

	We compute the parameters $\tau_{i}^{k}$, $i = 1 , 2$, by invoking \eqref{L:Technical:1} and 
	\eqref{L:Technical:2}, where we practically choose $\varepsilon = 0$. Hence, for the first, 
	completely nonconvex block, we have
	\begin{equation*}
		\delta_{1}^{k} = \frac{\alpha_{1}^{k} + \beta_{1}^{k}}{1 - 2\alpha_{1}^{k}}L_{1}(x_{2}^{k}), 
		\quad \tau_{1}^{k} = \frac{\delta_{1}^{k} + \left(1 + \beta_{1}^{k}\right)L_{1}(x_{2}^{k})}{1 -
		\alpha_{1}^{k}} \quad \Rightarrow \quad \tau_{1}^{k} = \frac{1 + 2\beta_{1}^{k}}{1 - 2\alpha_{1}
		^{k}}L_{1}(x_{2}^{k}),
	\end{equation*}
	from which it directly follows that $\alpha \in \left[0 , 0.5\right)$.
\medskip

	For the second block, where the nonsmooth function is convex we follow Remark \ref{R:StepConvex} and 
	invoke \eqref{R:StepConvex:1} and \eqref{R:StepConvex:2} to obtain
	\begin{equation*}
		\delta_{2}^{k} = \frac{\alpha_{2}^{k} + 2\beta_{2}^{k}}{2\left(1 - 2\alpha_{2}^{k}\right)}L_{2}
		(x_{1}^{k + 1}), \quad \tau_{2}^{k} = \frac{\delta_{2}^{k} + \left(1 + \beta_{2}^{k}\right)L_{2}
		(x_{1}^{k + 1})}{2 - \alpha_{2}^{k}} \quad \Rightarrow \quad \tau_{2}^{k} = \frac{1 + 2\beta_{2}
		^{k}}{2\left(1 - \alpha_{2}^{k}\right)}L_{2}(x_{1}^{k + 1}).
	\end{equation*}
	Comparing this parameter to the parameter of the first block we see that now $\alpha \in \left[0 , 
	1\right)$ and the value of $\tau$ is smaller by a factor of $2$. Hence, convexity in the nonsmooth 
	function allows for twice larger steps.

	\begin{figure}[ht!]
  		\centering
  		\subfigure[$s=50\%$]{\includegraphics[width=0.3\textwidth]{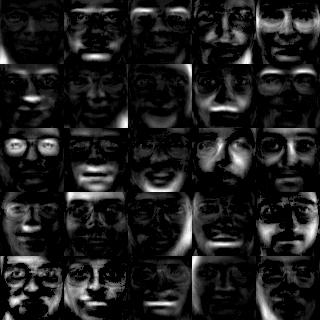}}\hfill
  		\subfigure[$s=33\%$]{\includegraphics[width=0.3\textwidth]{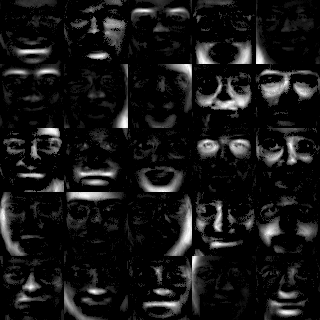}}\hfill
  		\subfigure[$s=25\%$]{\includegraphics[width=0.3\textwidth]{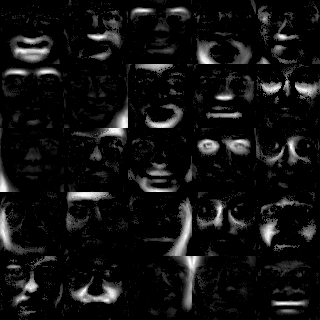}}
  		\caption{$25$ basis faces using different sparsity settings. A sparsity of $s = 25\%$ means that 
  		each basis face contains only $25\%$ non-zero pixels. Clearly, stronger sparsity leads to a more 
  		compact representation.}
  		\label{fig:basis-faces}
	\end{figure}

	In Tables~\ref{tab:nnmf-exact} and~\ref{tab:nnmf-bt}, we report the performance of the iPALM 
	algorithm for different settings of the inertial parameters $\alpha_{i}^{k}$ and $\beta_{i}^{k}$, $i 
	= 1 , 2$ and $k \in \nn$. Since we are solving a nonconvex problem, we report the values of the 
	objective function after a certain number of iterations ($100, 500, 1000$ and $5000$). As already 
	mentioned, setting $\alpha_{i} = \beta_{i} = 0$, $i = 1 , 2$, reverts the proposed iPALM algorithm 
	to the PALM algorithm \cite{BST2014}. In order to estimate the (local) Lipschitz constants $L_{1}
	(x_{2}^{k})$ and $L_{2} (x_{1}^{k + 1})$ we computed the exact values of the Lipschitz constants by
	computing the largest eigenvalues of $C^{k}(C^{k})^{T}$ and $(B^{k})^{T}B^{k}$, respectively. The 
	results are given in table \ref{tab:nnmf-exact}. Furthermore, we also implemented a standard 
	backtracking procedure, see for example \cite{BT09,OCBP2014}, which makes use of the descent lemma 
	in order to estimate the value of the (local) Lipschitz constant. In terms of iterations of iPALM, 
	the backtracking procedure generally leads to a better overall performance but each iteration also 
	takes more time compared to the exact computation of the Lipschitz constants. The results based on
	backtracking are shown in table~\ref{tab:nnmf-bt}.

	\begin{table}[t!]
		\begin{tabular}{c|c|c|c|c||c}
    	    		\toprule  
    	    		& $K = 100$ & $K = 500$ & $K = 1000$ & $K = 5000$ & time (s)\\           
    	    		\midrule
    	    		$\alpha_{1,2} = \beta_{1,2} = 0.0$ & 12968.17 & 7297.70  & 5640.11 & 4088.22 & 196.63 \\
    	    		$\alpha_{1,2} = \beta_{1,2} = 0.2$ & 12096.91 & 8453.29  & 6810.63 & 4482.00 & 190.47 \\
    	    		$\alpha_{1,2} = \beta_{1,2} = 0.4$ & 12342.81 & 11496.57 & 9277.11 & 5617.02 & 189.27 \\
    	    		\midrule
    	    		$\alpha_{1} = \beta_{1} = 0.2, \alpha_{2} = \beta_{2} = 0.4$ & 12111.55 & 8488.35 & 6822.95 
    	    		& 4465.59 & 201.05 \\
    	    		$\alpha_{1} = \beta_{1} = 0.4, \alpha_{2} = \beta_{2} = 0.8$ & 12358.00 & 11576.72 & 9350.37 
    	    		& 5593.84 & 200.11 \\
        		\midrule
       		$\alpha_{1,2}^{k} = \beta_{1,2}^{k} = (k - 1)/(k + 2)$ & 5768.63 & 3877.41 & 3870.98 & 
       		3870.81 & 186.62 \\
    	    		\bottomrule
		\end{tabular}  
  	  	\caption{Values of the objective function of the sparse NMF problem after $K$ iterations, using
  	  	different settings for the inertial parameters $\alpha_{i}$ and $\beta_{i}$, $i = 1 , 2$ and 
  	  	using computation of the exact Lipschitz constant.}
  	  	\label{tab:nnmf-exact}
	\end{table}

	\begin{table}[t!]
		\centering
  	  	\begin{tabular}{c|c|c|c|c||c}
    	    		\toprule  
    	    		& $K = 100$ & $K = 500$ & $K = 1000$ & $K = 5000$ & time (s)\\           
    	    		\midrule
    	    		$\alpha_{1,2} = \beta_{1,2} = 0.0$ & 8926.23 & 5037.89 & 4356.65 & 4005.53 & 347.17 \\
    	    		$\alpha_{1,2} = \beta_{1,2} = 0.2$ & 8192.71 & 4776.64 & 4181.40 & 4000.41 & 349.42 \\
    	    		$\alpha_{1,2} = \beta_{1,2} = 0.4$ & 8667.62 & 4696.64 & 4249.57 & 4060.95 & 351.78 \\
    	    		\midrule
    	    		$\alpha_{1} = \beta_{1} = 0.2, \alpha_{2} = \beta_{2} = 0.4$ & 8078.14 & 4860.74 & 4274.46 & 
    	    		3951.28 & 353.53 \\
	    	    $\alpha_{1} = \beta_{1} = 0.4, \alpha_{2} = \beta_{2} = 0.8$ & 8269.27 & 4733.76 & 4243.29 & 
	    	    4066.63 & 357.35 \\       
    	    		\midrule
    	    		$\alpha_{1,2}^{k} = \beta_{1,2}^{k} = (k-1)/(k+2)$ & 5071.71 & 3902.91 & 3896.40 & 3869.13 & 			347.90 \\    
    	    		\midrule
    	    		iPiano ($\beta = 0.4$) & 14564.65 & 12200.78 & 11910.65 & 7116.22 & 258.42 \\
    	    		\bottomrule
  	  	\end{tabular}  
  	  	\caption{Values of the objective function of the sparse NMF problem after $K$ iterations, using
		different settings for the inertial parameters $\alpha_{i}$ and $\beta_{i}$, $i = 1 , 2$ and 
		using backtracking to estimate the Lipschitz constant.} 
 		\label{tab:nnmf-bt}
	\end{table}

	We tried the following settings for the inertial parameters.
	\begin{itemize}
		\item \textbf{Equal:} Here we used the same inertial parameters for both the first and the 
			second block. In case all parameters are set to be zero, we recover PALM. We observe that 
			the use of the inertial parameters can speed up the convergence but for too large inertial 
			parameters we also observe not as good results as in the case that we use PALM, \ie no 
			inertial is used. Since the problem is highly nonconvex the final value of the objective 
			function can be missleading since it could correspond to a bad stationary point.
		\item \textbf{Double:} Since the nonsmooth function of the second block is convex, we can take 
			twice larger inertial parameters. We observe an additional speedup by taking twice larger 
			inertial parameters. Again, same phenomena occur here, too large inertial parameters yields
  			inferior performances of the algorithm.
		\item \textbf{Dynamic:} We also report the performance of the algorithm in case we choose 
			dynamic inertial parameters similar to accelerated methods in smooth convex optimization
			\cite{N04}. We use $\alpha_{i}^{k} = \beta_{i}^{k} = \left(k - 1\right)/\left(k + 2\right)$, 
			$i = 1 , 2$, and we set the parameters $\tau_{1}^{k} = L_{1}(x_{2}^{k})$ and $\tau_{2}^{k} = 
			L_{2}(x_{1}^{k + 1})$. One can see that this setting outperforms the other settings by a 
			large margin. Although our current convergence analysis does not support this setting, it 
			shows the great potential of using inertial algorithms when tackling nonconvex optimization 
			problems. The investigation of the convergence properties in this setting will be subject to 
			future research.
	\end{itemize}
	Finally, we also compare the proposed algorithm to the iPiano algorithm \cite{OCBP2014}, which is 
	similar to the proposed iPALM algorithm but does not make use of the block structure. Note that in 
	theory, iPiano is not applicable since the gradient of the overall problem is not Lipschitz 
	continuous. However, in practice it turns out that using a backtracking procedure to determine the 
	Lipschitz constant of the gradient is working and hence we show comparisons. In the iPiano 
	algorithm, the step-size parameter $\tau$ ($\alpha$ in terms of \cite{OCBP2014}) was set as
	\begin{equation*}
		\tau \leq \frac{1 - 2\beta}{L},
	\end{equation*}
	from which it follows that the inertial parameter $\beta$ can be chosen in the interval $\left[0 , 
	0.5\right)$. We used $\beta = 0.4$ since it give the best results in our tests. Observe that the 
	performance of iPALM is much better compared to the performance of iPiano. In terms of CPU time, one 
	iteration of iPiano is clearly slower than one iteration of iPALM using the exact computation of the 
	Lipschitz constant, because the backtracking procedure is computationally more demanding than the 
	computation of the largest eigenvalues of $C^{k}(C^{k})^{T}$ and $(B^{k})^{T}B^{k}$. Comparing the 
	iPiano algorithm to the version of iPALM which uses backtracking, one iteration of iPiano is faster 
	since iPALM needs to backtrack the Lipschitz constants for both of the two blocks.

\subsection{Blind Image Deconvolution}
	In our second example, we consider the well-studied (yet challenging) problem of blind image 
	deconvolution (BID). Given a blurry and possibly noisy image $f \in \rr^{M}$ of $M = m_{1} \times 
	m_{2}$ pixels, the task is to recover both a sharp image $u$ of the same size and the unknown point 
	spread function $b \in \rr^{N}$, which is a small $2D$ blur kernel of size $N = n_{1} \times n_{2}$ 
	pixels. We shall assume that the blur kernel is normalized, that is $b \in \Delta^{N}$, where $
	\Delta^{N}$ denotes the standard unit simplex defined by
	\begin{equation}
  		\Delta^{N} = \left\{ b \in \rr^{N} : \, b_{i} \geq 0, \; i = 1 , 2 , \ldots , N, \; \sum_{i = 1}
    		^{N} b_{i} = 1 \right\}.  
	\end{equation}
	Furthermore, we shall assume that the pixel intensities of the unknown sharp image $u$ are 
	normalized to the interval $\left[0 , 1\right]$, that is $u \in U^{M}$, where
	\begin{equation}
  		U^{M} = \left\{ u \in \rr^{M} : \, u_{i} \in \left[0 , 1\right] , \; i = 1 , 2, \ldots , M 
  		\right\}.
	\end{equation}
	We consider here a classical blind image deconvolution model (see, for example, 
	\cite{perrone_emmcvpr2015}) defined by
	\begin{equation}
  		\min_{u , b} \left\{ \sum_{p = 1}^{8} \phi\left(\nabla_{p}u\right) + \frac{\lambda}{2}\norm{u
      \ast_{m_{1} , m_{2}}b - f}^{2} : \, u \in U^{M}, \; b \in \Delta^{N} \right\}.
	\end{equation}
	The first term is a regularization term which favors sharp images and the second term is a data 
	fitting term that ensures that the recovered solution approximates the given blurry image. The 
	parameter $\lambda > 0$ is used to balance between regularization and data fitting. The linear 
	operators $\nabla_{p}$ are finite differences approximation to directional image gradients, which in 
	implicit notation are given by
	\begin{align*}
  		\left(\nabla_{1}u\right)_{i , j} & = u_{i + 1 , j} - u_{i , j}, &
  		\left(\nabla_{2}u\right)_{i , j} & = u_{i , j + 1} - u_{i , j}, \\
  		\left(\nabla_{3}u\right)_{i , j} & = \frac{u_{i + 1 , j + 1} - u_{i ,j }}{\sqrt{2}}, &
  		\left(\nabla_{4}u\right)_{i , j} & = \frac{u_{i + 1 , j - 1} - u_{i , j}}{\sqrt{2}} \\
  		\left(\nabla_{5}u\right)_{i , j} & = \frac{u_{i + 2 , j + 1} - u_{i , j}}{\sqrt{5}}, &
  		\left(\nabla_{6}u\right)_{i , j} & = \frac{u_{i + 2 , j - 1} - u_{i , j}}{\sqrt{5}}, \\
  		\left(\nabla_{7}u\right)_{i , j} & = \frac{u_{i + 1 , j + 2} - u_{i , j}}{\sqrt{5}}, &
  		\left(\nabla_{8}u\right)_{i , j} & = \frac{u_{i - 1 , j + 2} - u_{i , j}}{\sqrt{5}},
	\end{align*}
	for $1 \leq i \leq m_{1}$ and $1 \leq j \leq m_{2}$. We assume natural boundary conditions, that is 
	$\left(\nabla_{p}\right)_{i , j} = 0$, whenever the operator references a pixel location that lies 
	outside the domain. The operation $u\ast_{m_{1} , m_{2}}b$ denotes the usual 2D modulo - $m_{1} , 
	m_{2}$ discrete circular convolution operation defined by (and interpreting the image $u$ and the 
	blur kernel $b$ as 2D arrays)
	\begin{equation}
  		\left(u\ast_{m_{1} , m_{2}} b\right)_{i , j} = \sum_{k = 0}^{n_{1}}\sum_{l = 0}^{n_{2}} b_{k , 
  		l}\,u_{(i - k)_{\mathrm{mod} \, m_{1}} , (j - l)_{\mathrm{mod} \, m_{2}}}, \quad 1 \leq i \leq 
  		m_{1}, \; 1 \leq j \leq m_{2}.
	\end{equation}	
	For ease the notation, we will rewrite the 2D discrete convolutions as the matrix vector products of 
	the form
	\begin{equation}
  		v = u\ast_{m_{1} , m_{2}} b \Leftrightarrow v = K\left(b\right)u \Leftrightarrow v = K\left(u  
  		\right)b,
	\end{equation}
	where $K\left(b\right) \in \rr^{M \times M}$ is a sparse matrix, where each row holds the values of 
	the blur kernel $b$, and $K\left(u\right) \in \rr^{M \times N}$ is a dense matrix, where each column 
	is given by a circularly shifted version of the image $u$. Finally, the function $\phi\left(\cdot
	\right)$ is a differentiable robust error function, that promotes sparsity in its argument. For a 
	vector $x \in \rr^M$, the function $\phi$ is defined as
	\begin{equation}
  		\phi\left(x\right) = \sum_{i = 1}^{M} \log\left(1 + \theta x_{i}^{2}\right),
	\end{equation}
	where $\theta > 0$ is a parameter. Here, since the argument of the function are image gradients, the
	function promotes sparsity in the edges of the image. Hence, we can expect that sharp images result 
	in smaller values of the objective function than blurry images. For images that can be well 
	described by piecewise constant functions (see, for example, the books image in Figure
	\ref{fig:bid}), such sparsity promoting function might be well suited to favor sharp images, but we 
	would like to stress that this function could be a bad choice for textured images, since a sharp
	image usually has much stronger edges than the blurry image. This often leads to the problem that 
	the trivial solution ($b$ being the identity kernel and $u = f$) has a lower energy compared to the 
	true solution.  
\medskip

	In order to apply the iPALM algorithm, we identify the following functions
	\begin{equation}
  		H\left(u , b\right) = \sum_{p = 1}^{8} \phi\left(\nabla_{p}u\right) + \frac{\lambda}{2}\norm{u
   	 	\ast_{m_{1} , m_{2}}b - f}^{2},
	\end{equation}
	which is smooth with block Lipschitz continuous gradients given by
	\begin{align*}
  		\nabla_{u} H\left(u , b\right) & = 2\theta\sum_{p = 1}^{8} \nabla_{p}^{T}\mathrm{vec} 
  		\left(\frac{\left(\nabla_{p}u\right)_{i , j}}{1 + \theta\left(\nabla_{p}u\right)_{i , j}^{2}}
  		\right)_{i , j = 1}^{m_{1} , m_{2}} + \lambda K^{T}\left(b\right)\left(K\left(b\right)u - f
  		\right), \\
  		\nabla_{b} H\left(u , b\right) & = \lambda K^{T}\left(u\right)\left(K\left(u\right)b - f\right),
	\end{align*}
	where the operation $\mathrm{vec}\left(\cdot\right)$ denotes the formation of a vector from the 
	values passed to its argument. The nonsmooth function of the first block is given by
	\begin{equation}
  		f_{1}\left(u\right) = 
  		\begin{cases}
    			0, \quad u \in U^{M}, \\
    			\infty, \,\,\, \text{else},
  		\end{cases}
	\end{equation}
	and the proximal map with respect to $f_{1}$ is computed as
	\begin{equation}
  		u = \prox_{f_{1}}\left(\hat{u}\right) \Leftrightarrow u_{i , j} = \max\left\{ 0 , \min(1 , 
    		\hat{u_{i , j}}) \right\}.
	\end{equation}
	The nonsmooth function of the second block is given by the indicator function of the unit simplex 
	constraint, that is,
	\begin{equation}
  		f_{2}\left(b\right) = 
  		\begin{cases}
    			0, \quad b \in \Delta^{N}, \\
    			\infty, \,\,\, \text{else}.
 		\end{cases}
	\end{equation}
	In order to compute the proximal map with respect to $f_{2}$, we use the algorithm proposed in 
	\cite{Duchietal} which computes the projection onto the unit simplex in $O(N\log N)$ time.

	\begin{figure}[t!]
  		\centering 
  		\subfigure[original books image]{\includegraphics[width=0.3\textwidth]{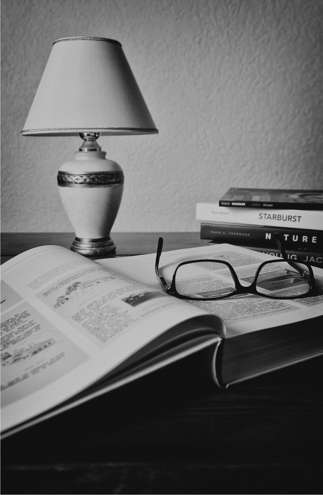}}
  		\hspace{1cm}
  		\subfigure[convolved image]{\includegraphics[width=0.3\textwidth]{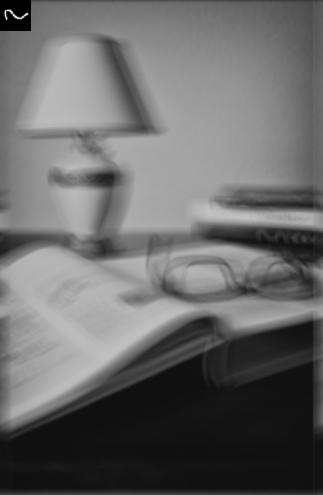}} \\
  		\subfigure[$\alpha_{1 , 2} = \beta_{1 , 2} = 0$]{\includegraphics[width=0.3\textwidth]{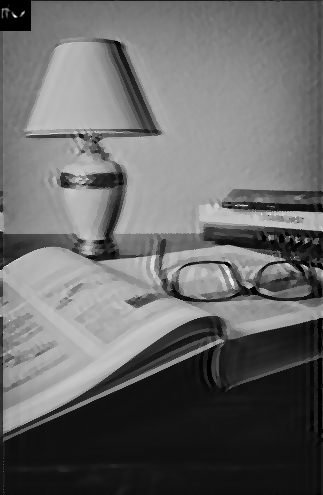}}
  		\hfill 
  		\subfigure[$\alpha_{1 , 2} = \beta_{1 , 2} = 0.4$]{\includegraphics[width=0.3\textwidth]{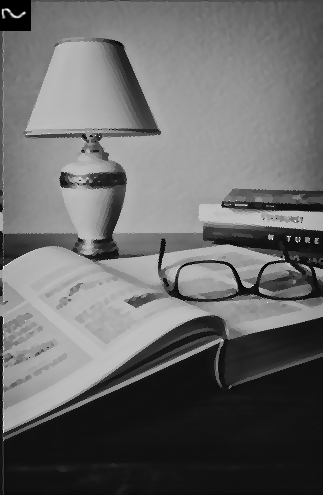}}
  		\hfill 
  		\subfigure[$\alpha_{1 , 2} = \beta_{1 , 2} = \frac{k- 1 }{k + 2}$]	{\includegraphics[width=0.3\textwidth]{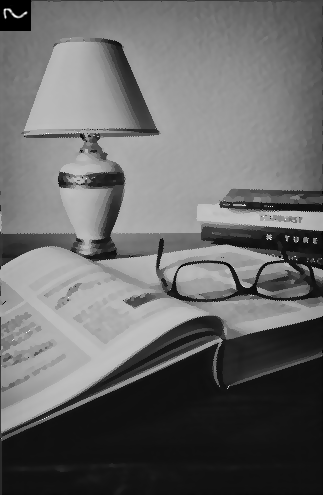}}
  		\caption{Results of blind deconvolution using $K = 5000$ iterations and different settings of 
  		the inertial parameters. The result without inertial terms (\ie $\alpha_{i} = \beta_{i} = 0$ for 
  		$i = 1 , 2$) is significantly worse compared to the result using inertial terms. The best 
  		results are obtained using the dynamic choice of the inertial parameters.}
  		\label{fig:bid}
	\end{figure}

	We applied the BID problem to the books image of size $m_{1} \times m_{2} = 495 \times 323$ pixels 
	(see Figure \ref{fig:bid}). We set $\lambda = 10^{6}$, $\theta = 10^{4} $ and generated the blurry 
	image by convolving it with a s-shaped blur kernel of size $n_{1} \times n_{2} = 31 \times 31$ 
	pixels. Since the nonsmooth functions of both blocks are convex, we can set the parameters 
	as (compare to the first example)
	\begin{equation} \label{eq:tau-bid}
  		\tau_{1}^{k} = \frac{1 + 2\beta_{1}^{k}}{2\left(1 - \alpha_{1}^{k}\right)}L_{1}(x_{2}^{k}) \quad 
  		\text{and} \quad \tau_{2}^{k} = \frac{1 + 2\beta_{2}^{k}}{2\left(1 - \alpha_{2}^{k}\right)}L_{2}
  		(x_{1}^{k + 1}).
	\end{equation}
	To determine the values of the (local) Lipschitz constants, we used again a backtracking scheme
	\cite{BT09,OCBP2014}. In order to avoid the trivial solution (that is, $u = f$ and $b$ being the 
	identity kernel) we took smaller descent steps in the blur kernel which was realized by multiplying 
	$\tau_{2}^{k}$, $k \in \nn$, by a factor of $c = 5$. Note that this form of ``preconditioning'' does 
	not violate any step-size restrictions as we can always take larger values for $\tau$ than the value 
	computed in \eqref{eq:tau-bid}.

	\begin{table}[t!]
  		\centering
  		\begin{tabular}{c|c|c|c|c||c}
    			\toprule  
    			& $K = 100$ & $K = 500$  & $K = 1000$ & $K = 5000$ & time (s)\\
    			\midrule
    			$\alpha_{1,2} = \beta_{1,2} = 0.0$ & 2969668.92 & 1177462.72 & 1031575.57 & 847268.70 & 
    			1882.63 \\
    			$\alpha_{1,2} = \beta_{1,2} = 0.4$ & 5335748.90 & 1402080.44 & 1160510.16 & 719295.30 & 
    			1895.61 \\
    			$\alpha_{1,2} = \beta_{1,2} = 0.8$ & 5950073.38 & 1921105.31 & 1447739.06 & 780109.56 & 
    			1888.25 \\
    			\midrule
    			$\alpha_{1,2}^{k} = \beta_{1,2}^{k} = (k - 1)/(k + 2)$ & 2014059.03 & 978234.23 & 683694.72 
    			& 678090.51 & 1867.19 \\
		    \bottomrule
  		\end{tabular}   
  	 	\caption{Values of the objective function of the BID problem, after $K$ iterations and using 
  	 	different settings for the inertial parameters $\alpha_{i}$ and $\beta_{i}$ for $i = 1 , 2$.}		
  		\label{tab:bid}
	\end{table}

	Table \ref{tab:bid} shows an evaluation of the iPALM algorithm using different settings of the 
	inertial parameters $\alpha_{i}$ and $\beta_{i}$ for $i = 1 , 2$. First, we observe that the use of 
	inertial parameters lead to higher values of the objective function after a smaller number of 
	iterations. However, for a larger number of iterations the use of inertial forces leads to 
	significantly lower values. Again, the use of dynamic inertial parameters together with the 
	parameter $\tau_{1}^{k} = L_{1}(x_{2}^{k})$ and $\tau_{2}^{k} = L_{2}(x_{1}^{k + 1})$ leads to the 
	best overall performance. In Figure \ref{fig:bid} we show the results of the blind deconvolution 
	problem. One can see that the quality of the recovered image as well as the recovered blur kernel is 
	much better using inertial forces. Note that the recovered blur kernel is very close to the true 
	blur kernel but the recovered image appears slightly more piecewise constant than the original 
	image.

\subsection{Convolutional LASSO}

	In our third experiment we address the problem of sparse approximation of an image using dictionary 
	learning. Here, we consider the convolutional LASSO model~\cite{Zeiler10}, which is an interesting 
	variant of the well-known patch-based LASSO model~\cite{Olshausen97,AharonEladBruckstein-KSVD} for 
	sparse approximations. The convolutional model inherently models the transnational invariance of 
	images, which can be considered as an advantage over the usual patch-based model which treats every 
	patch independently.
\medskip

	The idea of the convolutional LASSO model is to learn a set of small convolution filters $d_{j} \in 
	\rr^{l \times l}$, $j = 1 , 2 , \ldots , p$, such that a given image $f \in \rr^{m \times n}$ can be 
	written as $f \approx \sum_{j = 1}^{p} d_{j}\ast_{m,n} v_{j}$, where $\ast_{m,n}$ denotes again the
	$2D$ modulo $m , n$ discrete circular convolution operation and $v_{j} \in \rr^{m \times n}$ are the 
	corresponding coefficient images which are assumed to be sparse. In order to make the convolution 
	filters capture the high-frequency information in the image, we fix the first filter $d_{1}$ to be a 
	Gaussian (low pass) filter $g$ with standard deviation $\sigma_{l}$ and we set the corresponding 
	coefficient image $v_{j}$ equal to the initial image $f$. Furthermore, we assume that the remaining 
	filters $d_{j}$, $j = 2 , 3 , \ldots , p$ have zero mean as well as a $\ell_{2}$-norm less or equal 
	to $1$. In order to impose a sparsity prior on the coefficient images $v_{j}$ we make use of the 
	$\ell_{1}$-norm. The corresponding objective function is hence given by
	\begin{align} \label{eq:convolutional-lasso}
  		\min_{(d_j)_{j=1}^{p},(v_j)_{j=1}^{p}} \sum_{j = 1}^{p} \lambda\norm{v_{j}}_{1} + \frac{1}{2}
  		\norm{\sum_{j = 1}^{p} d_{j}\ast_{m,n} v_{j} - f}_{2}^{2}, \\
  		\text{s.t.} \, d_{1} = g, \, v_{1} = f \, \sum_{a,b=1}^{l} (d_j)_{a,b} = 0, \, \norm{d_{j}}_{2} 
  		\leq 1, \, j = 2 , 3 , \ldots , p, \nonumber
	\end{align}
	where the parameter $\lambda > 0$ is used to control the degree of sparsity. It is easy to see that 
	the convolutional LASSO model nicely fits to the class of problems that can be solved using the 
	proposed iPALM algorithm. We leave the details to the interested reader. In order to compute the 
	(local) Lipschitz constants we again made use of a backtracking procedure and the parameters 
	$\tau_{i}^{k}$, $i = 1 , 2$ were computed using~\eqref{eq:tau-bid}.

	\begin{table}[t!]
  		\centering
  		\begin{tabular}{c|c|c|c|c||c}
    			\toprule  
    			& $K = 100$ & $K = 200$ & $K = 500$ & $K = 1000$ & time (s) \\           
    			\midrule
    			$\alpha_{1,2} = \beta_{1,2} = 0.0$ & 336.13 & 328.21 & 322.91 & 321.12 & 3274.97 \\
    			$\alpha_{1,2} = \beta_{1,2} = 0.4$ & 329.20 & 324.62 & 321.51 & 319.85 & 3185.04 \\
    			$\alpha_{1,2} = \beta_{1,2} = 0.8$ & 325.19 & 321.38 & 319.79 & 319.54 & 3137.09 \\
    			\midrule    
    			$\alpha_{1,2}^{k} = \beta_{1,2}^{k} = (k - 1)/(k + 2)$ & 323.23 & 319.88 & 318.64 & 318.44 & 
    			3325.37 \\
    			\bottomrule
  		\end{tabular}  
  		\caption{Values of the objective function for the convolutional LASSO model using different 
  		settings of the inertial parameters.} 
  		\label{tab:dile}
	\end{table}

	We applied the convolutional LASSO problem to the Barbara image of size $512 \times 512$ pixels, 
	which is shown in Figure \ref{fig:dile}. The Barbara image contains a lot of stripe-like texture and 
	hence we expect that the learned convolution filters will contain these characteristic structures. 
	In our experiment, we learned a dictionary made of $81$ filter kernels, each of size $9 \times 9$ 
	pixels. The regularization parameter $\lambda$ was set to $\lambda = 0.2$. From Figure 
	\ref{fig:dile} it can be seen that the learned convolution filters indeed contain stripe-like 
	structures of different orientations but also other filters that are necessary to represent the 
	other structures in the image. Table \ref{tab:dile} summarizes the performance of the iPALM 
	algorithm for different settings of the inertial parameters. From the results, one can see that 
	larger settings of the inertial parameters lead to a consistent improvement of the convergence 
	speed. Again, using a dynamic choice of the inertial parameters clearly outperforms the other 
	settings.
\medskip

	For illustration purposes we finally applied the learned dictionary to denoise a noisy variant of 
	the same Barbara image. The noisy image has been generated by adding zero-mean Gaussian noise with 
	standard deviation $\sigma = 0.1$ to the original image. For denoising we use the previously learned 
	dictionary $d$ and minimizing the convolutional LASSO problem only with respect to the coefficient 
	images $v$. Note that this is a convex problem and hence,it can be efficiently minimized using for 
	example the FISTA algorithm \cite{BT09}. The denoised image is again shown in Figure \ref{fig:dile}. 
	Observe that the stripe-like texture is very well preserved in the denoised image.

	\begin{figure}[t!]
  		\centering 
  		\subfigure[Original Barbara image]{\includegraphics[width=0.45\textwidth]{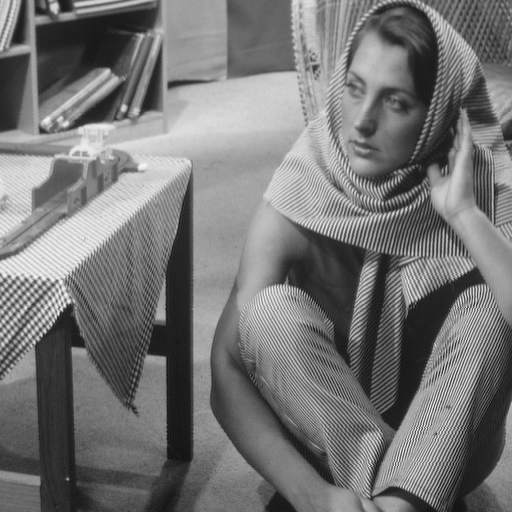}}
   		\hfill
  		\subfigure[Learned $9 \times 9 $ dictionary]{\includegraphics[width=0.3\textwidth]{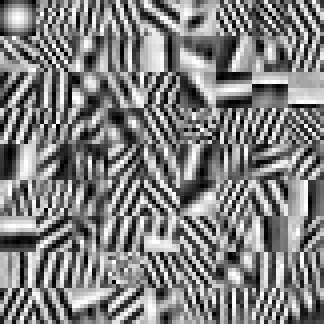}}\hspace*{1.1cm}\\
  		\subfigure[Noisy Barbara image ($\sigma = 0.1$)]{\includegraphics[width=0.45\textwidth]{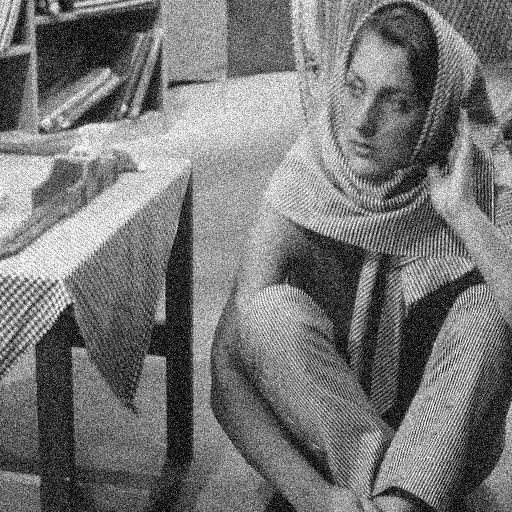}}
  		\hfill 
  		\subfigure[Denoised Barbara image (PSNR=28.33)]{\includegraphics[width=0.45\textwidth]{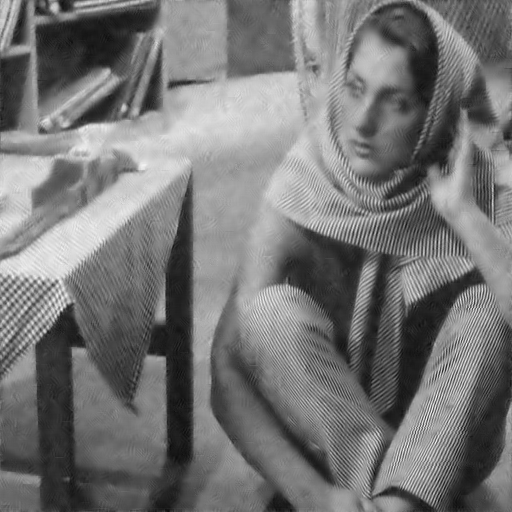}}
  		\caption{Results of dictionary learning using the convolutional LASSO model. Observe that the 
  		learned dictionary very well captures the stripe-like texture structures of the Barbara image.}			\label{fig:dile}
	\end{figure}

\section{Conclusion}

        In this paper we proposed iPALM which an inertial variant of the Proximal Alternating Linearized 
        Minimization (PALM) method proposed in \cite{BST2014} for solving a broad class of nonconvex and 
        nonsmoooth optimization problems consisting of block-separable nonsmooth, nonconvex functions 
        with easy to compute proximal mappings and a smooth coupling function with block-Lipschitz 
        continuous gradients. We studied the convergence properties of the algorithm and provide bounds 
        on the inertial and step-size parameters that ensure convergence of the algorithm to a critical 
        point of the problem at hand. In particular, we showed that in case the objective function 
        satisfies the Kurdyka-{\L}ojasiewicz (KL) property, we can obtain finite length property of the 
        generated sequence of iterates. In several numerical experiments we show the advantages of the 
        proposed algorithm to minimize a number of well-studied problems and image processing and 
        machine learning. In our experiments we found that choosing the inertial and step-size 
        parameters dynamically, as pioneered by Nesterov~\cite{N83}, leads to a significant performance 
        boost, both in terms of convergence speed and convergence to a ``better'' critical point of the 
        problem. Our current convergence theory does not support this choice of parameters but 
        developing a more general convergence theory will be interesting and a subject for future 
        research.

\section{Appendix A: Proof of Lemma \ref{L:Technical}} \label{A:Technical}
	We first recall the result that should be proved.
	\begin{lemma}
		Consider the functions $g : \rr_{+}^{5} \rightarrow \rr$ and $h : \rr_{+}^{5} \rightarrow \rr$ 
		defined as follow
		\begin{align*}
			g\left(\alpha , \beta , \delta , \tau , L\right) & = \tau\left(1 - \alpha\right) - \left(1 + 
			\beta\right)L - \delta, \\
			h\left(\alpha , \beta , \delta , \tau , L\right) & = \delta - \tau\alpha - L\beta.
		\end{align*}
		Let $\varepsilon > 0$ and $\bar{\alpha} > 0$ be two real numbers for which $0 \leq \alpha \leq 
		\bar{\alpha} < 0.5\left(1 - \varepsilon\right)$. Assume, in addition, that $0 \leq L \leq 
		\lambda$ for some $\lambda > 0$ and $0 \leq \beta \leq \bar{\beta}$ with $\bar{\beta} > 0$. If
		\begin{align}
			\delta_{\ast} & = \frac{\bar{\alpha} + \bar{\beta}}{1 - \varepsilon - 2\bar{\alpha}}\lambda, 
			\label{L:Technical:11} \\
			\tau_{\ast} & = \frac{\left(1 + \varepsilon\right)\delta_{\ast} + \left(1 + \beta\right)L}
			{1 - \alpha}, \label{L:Technical:21}
		\end{align}
		then $g\left(\alpha , \beta , \delta_{\ast} , \tau_{\ast} , L\right) = 
		\varepsilon\delta_{\ast}$ and $h\left(\alpha , \beta , \delta_{\ast} , \tau_{\ast} , L\right) 
		\geq \varepsilon\delta_{\ast}$.		
	\end{lemma}
	\begin{proof}
		From the definition of $g$ we immediately obtain that
		\begin{equation*}
			g\left(\alpha , \beta , \delta_{\ast} , \tau_{\ast} , L\right) = \tau_{\ast}\left(1 - 
			\alpha\right) - \left(1 + \beta\right)L - \delta_{\ast} = \varepsilon\delta_{\ast},
		\end{equation*}
		this proves the first desired result. We next simplify $h\left(\alpha , \beta , \delta_{\ast} , 
		\tau_{\ast} , L\right) - \varepsilon\delta_{\ast}$ as follows
		\begin{align*}
        		h\left(\alpha , \beta , \delta_{\ast} , \tau_{\ast} , L\right) - \varepsilon\delta_{\ast} & 
        		= \left(1 - \varepsilon\right)\delta_{\ast} - \tau_{\ast}\alpha - L\beta \nonumber \\
			& = \left(1 - \varepsilon\right)\delta_{\ast} - \frac{\left(1 + \varepsilon\right)
			\delta_{\ast} + L\left(1 + \beta\right)}{1 - \alpha}\alpha - L\beta \nonumber \\
			& = \left(1 - \varepsilon\right)\delta_{\ast} - \frac{\left(1 + \varepsilon\right)
			\delta_{\ast}\alpha + L\left(1 + \beta\right)\alpha + L\beta\left(1 - \alpha\right)}{1 - 
			\alpha} \nonumber \\
			& = \left(1 - \varepsilon\right)\delta_{\ast} - \frac{\left(1 + \varepsilon\right)
			\delta_{\ast}\alpha + L\left(\alpha + \beta\right)}{1 - \alpha}.
        	\end{align*}	        
        	Thus, we only remain to show that
		\begin{equation*}
        		\left(1 - \varepsilon\right)\delta_{\ast} - \frac{\left(1 + \varepsilon\right)\delta_{\ast}
        		\alpha + L\left(\alpha + \beta\right)}{1 - \alpha} \geq 0.
        	\end{equation*}	    
        	Indeed, simple manipulations yields the following equivalent inequality    
		\begin{equation} \label{L:Technical:41}
        		\left(1 - \varepsilon - 2\alpha\right)\delta_{\ast} \geq L\left(\alpha + \beta\right).
        	\end{equation}	    
        	Using now \eqref{L:Technical:11} and the facts that $\alpha \leq {\bar \alpha}$, $\beta \leq 
        	\bar{\beta}$ and $0 < L \leq \lambda$ we obtain that \eqref{L:Technical:41} holds true. This 
        	completes the proof of the lemma.
	\end{proof}
\bibliographystyle{plain}
\bibliography{notes-1}

\def\cprime{$'$}
\begin{thebibliography}{10}

\bibitem{AharonEladBruckstein-KSVD}
M.~Aharon, M.~Elad, and A.~Bruckstein.
\newblock {K-SVD}: An algorithm for designing overcomplete dictionaries for
  sparse representation.
\newblock {\em Signal Processing, IEEE Transactions on}, 54(11):4311--4322,
  2006.

\bibitem{AlvarezAttouch2001}
F.~Alvarez and H.~Attouch.
\newblock An inertial proximal method for maximal monotone operators via
  discretization of a nonlinear oscillator with damping.
\newblock {\em Set-Valued Analysis}, 9(1-2):3--11, 2001.

\bibitem{AB2009}
H.~Attouch and J.~Bolte.
\newblock On the convergence of the proximal algorithm for nonsmooth functions
  involving analytic features.
\newblock {\em Math. Program.}, 116(1-2, Ser. B):5--16, 2009.

\bibitem{ABRS2010}
H.~Attouch, J.~Bolte, P.~Redont, and A.~Soubeyran.
\newblock Proximal alternating minimization and projection methods for
  nonconvex problems: an approach based on the {K}urdyka-{L} ojasiewicz
  inequality.
\newblock {\em Math. Oper. Res.}, 35(2):438--457, 2010.

\bibitem{ABS2013}
H.~Attouch, J.~Bolte, and B.~F. Svaiter.
\newblock Convergence of descent methods for semi-algebraic and tame problems:
  proximal algorithms, forward-backward splitting, and regularized
  {G}auss-{S}eidel methods.
\newblock {\em Math. Program.}, 137(1-2, Ser. A):91--129, 2013.

\bibitem{BT09}
A.~Beck and M.~Teboulle.
\newblock A fast iterative shrinkage-thresholding algorithm for linear inverse
  problems.
\newblock {\em SIAM J. Imaging Sci.}, 2(1):183--202, 2009.

\bibitem{BT89}
D.~P. Bertsekas and J.~N. Tsitsiklis.
\newblock {\em Parallel and Distributed Computation}.
\newblock Prentice-Hall International Editions, Englewood Cliffs, NJ, 1989.

\bibitem{BDL2006}
J.~Bolte, A.~Daniilidis, and A.~Lewis.
\newblock The {L}ojasiewicz inequality for nonsmooth subanalytic functions with
  applications to subgradient dynamical systems.
\newblock {\em SIAM J. Optim.}, 17(4):1205--1223, 2006.

\bibitem{BDLM10}
J.~Bolte, A.~Daniilidis, O.~Ley, and L.~Mazet.
\newblock Characterizations of {\l}ojasiewicz inequalities: subgradient flows,
  talweg, convexity.
\newblock {\em Trans. Amer. Math. Soc.}, 362(6):3319--3363, 2010.

\bibitem{BST2014}
J.~Bolte, S.~Sabach, and M.~Teboulle.
\newblock Proximal alternating linearized minimization for nonconvex and
  nonsmooth problems.
\newblock {\em Math. Program. Series A}, 146:459--494, 2014.

\bibitem{CW05}
P.~L. Combettes and V.~R. Wajs.
\newblock Signal recovery by proximal forward-backward splitting.
\newblock {\em Multiscale Model. Simul.}, 4(4):1168--1200, 2005.

\bibitem{Drori2013}
Y.~Drori and M.~Teboulle.
\newblock Performance of first-order methods for smooth convex minimization: a
  novel approach.
\newblock {\em Math. Program.}, 145(1-2, Ser. A):451--482, 2014.

\bibitem{Duchietal}
J.~Duchi, S.~Shalev-Shwartz, Y.~Singer, and T.Chandra.
\newblock Efficient projections onto the l1-ball for learning in high
  dimensions.
\newblock In {\em Proceedings of the 25th International Conference on Machine
  Learning}, ICML '08, pages 272--279, 2008.

\bibitem{Bot2015}
R.~.I.Bo{\c{T}} and E.~R. Csetnek.
\newblock An inertial tseng's type proximal algorithm for nonsmooth and
  nonconvex optimization problems.
\newblock {\em Journal of Optimization Theory and Applications}, pages 1--17,
  2015.

\bibitem{LS1999}
D.~D. Lee and H.~S. Seung.
\newblock Learning the parts of objects by nonnegative matrix factorization.
\newblock {\em Nature}, 401:788--–791, 1999.

\bibitem{Levin2009}
A.~Levin, Y.~Weiss, F.~Durand, and W.T. Freeman.
\newblock Understanding and evaluating blind deconvolution algorithms.
\newblock In {\em Computer Vision and Patter Recognition (CVPR)}, 2009.

\bibitem{LM79}
P.~L. Lions and B.~Mercier.
\newblock Splitting algorithms for the sum of two nonlinear operators.
\newblock {\em SIAM Journal on Applied Mathematics}, 16(6):964--979, 1979.

\bibitem{LiuNocedal89}
D.~C. Liu and J.~Nocedal.
\newblock On the limited memory {BFGS} method for large scale optimization.
\newblock {\em Mathematical Programming}, 45(1):503--528, 1989.

\bibitem{M2006-B}
B.~S. Mordukhovich.
\newblock {\em Variational analysis and generalized differentiation. {I}},
  volume 330 of {\em Grundlehren der Mathematischen Wissenschaften [Fundamental
  Principles of Mathematical Sciences]}.
\newblock Springer-Verlag, Berlin, 2006.
\newblock Basic theory.

\bibitem{M65}
J.~J. Moreau.
\newblock Proximit\'eet dualit\'e dans un espace hilbertien.
\newblock {\em Bull. Soc. Math. France}, 93:273--299, 1965.

\bibitem{N04}
Y.~Nesterov.
\newblock {\em Introductory Lectures on Convex Optimization}, volume~87 of {\em
  Applied Optimization}.
\newblock Kluwer Academic Publishers, Boston, MA, 2004.

\bibitem{N83}
Y.~E. Nesterov.
\newblock A method for solving the convex programming problem with convergence
  rate {$O(1/k\sp{2})$}.
\newblock {\em Dokl. Akad. Nauk SSSR}, 269(3):543--547, 1983.

\bibitem{Nocedal06}
J.~Nocedal and S.~J. Wright.
\newblock {\em Numerical Optimization}.
\newblock Springer, New York, 2nd edition, 2006.

\bibitem{OCBP2014}
P.~Ochs, Y.~Chen, T.~Brox, and T.Pock.
\newblock i{P}iano: inertial proximal algorithm for nonconvex optimization.
\newblock {\em SIAM J. Imaging Sci.}, 7(2):1388--1419, 2014.

\bibitem{Olshausen97}
B.~A. Olshausen and D.~J. Field.
\newblock Sparse coding with an overcomplete basis set: A strategy employed by
  v1?
\newblock {\em Vision Research}, 37(23):3311 -- 3325, 1997.

\bibitem{PT1994}
P.~Paatero and U.~Tapper.
\newblock Positive matrix factorization: A nonnegative factor model with
  optimal utilization of error estimates of data values.
\newblock {\em Environmetrics}, 5:111--–126, 1994.

\bibitem{Peharz}
R.~Peharz and F.~Pernkopf.
\newblock Sparse nonnegative matrix factorization with ℓ0-constraints.
\newblock {\em Neurocomputing}, 80(0):38 -- 46, 2012.

\bibitem{perrone_emmcvpr2015}
D.~Perrone, R.~Diethelm, and P.~Favaro.
\newblock Blind deconvolution via lower-bounded logarithmic image priors.
\newblock In {\em International Conference on Energy Minimization Methods in
  Computer Vision and Pattern Recognition (EMMCVPR)}, 2015.

\bibitem{Polyak64}
B.~T. Polyak.
\newblock Some methods of speeding up the convergence of iteration methods.
\newblock {\em USSR Computational Mathematics and Mathematical Physics},
  4(5):1--17, 1964.

\bibitem{B1964}
B.~T. Polyak.
\newblock Some methods of speeding up the convergence of iteration methods.
\newblock {\em \{U.S.S.R.\} Comput. Math. and Math. Phys.}, 4(5):1--17, 1964.

\bibitem{RW1998-B}
R.~T. Rockafellar and R.~J.-B. Wets.
\newblock {\em Variational analysis}, volume 317 of {\em Grundlehren der
  Mathematischen Wissenschaften [Fundamental Principles of Mathematical
  Sciences]}.
\newblock Springer-Verlag, Berlin, 1998.

\bibitem{orl}
F.~Samaria and A.~Harter.
\newblock Parameterisation of a stochastic model for human face identification.
\newblock In {\em WACV}, pages 138--142. IEEE, 1994.

\bibitem{SD2006}
S.~Sra and I.~S. Dhillon.
\newblock Generalized nonnegative matrix approximations with {B}regman
  divergences.
\newblock In Y.~Weiss, B.~Sch\"{o}lkopf, and J.C. Platt, editors, {\em Advances
  in Neural Information Processing Systems 18}, pages 283--290. MIT Press,
  2006.

\bibitem{Xu2014}
Y.~Xu and W.~Yin.
\newblock A globally convergent algorithm for nonconvex optimization based on
  block coordinate update.
\newblock Technical report, Arxiv preprint, 2014.

\bibitem{ZK93}
S.K. Zavriev and F.V. Kostyuk.
\newblock Heavy-ball method in nonconvex optimization problems.
\newblock {\em Computational Mathematics and Modeling}, 4(4):336--341, 1993.

\bibitem{Zeiler10}
M.~Zeiler, D.~Krishnan, G.~Taylor, and R.~Fergus.
\newblock Deconvolutional networks.
\newblock In {\em {IEEE} Conference on Computer Vision and Pattern Recognition
  (CVPR)}, pages 2528--2535, 2010.

\end{thebibliography}

\end{document}